\documentclass{amsart}
\usepackage{multirow}
\usepackage{enumerate}
\usepackage{amsmath}
\usepackage{amsthm}
\usepackage{mathrsfs}
\usepackage{booktabs,caption}
\usepackage[flushleft]{threeparttable}
\usepackage{amssymb}
\usepackage{mathtools}
\usepackage{color}
\usepackage{algorithm}
\usepackage{algorithmic}
\usepackage{cite,url}
\usepackage{latexsym}
\usepackage{booktabs}
\usepackage{bm}
\usepackage{graphicx}
\usepackage{longtable}
\usepackage[flushleft]{threeparttable}
\usepackage{textcomp}
\usepackage{subfigure}
\usepackage[toc,page]{appendix}
\usepackage{tabu}
\usepackage{enumitem}

\newcommand{\Rni}{\mathbb{R}^{n_{i}}}

\newcommand{\nfR}{\mathbb{R}}
\newcommand{\xpi}{x_i}

\newcommand{\xmi}{x_{-i}}
\newcommand{\umi}{u_{-i}}
\newcommand{\mbF}{\mbox{F}}

\newcommand{\fpi}{f_{i}}

\newcommand{\lvt}{\left[}
\newcommand{\rvt}{\right]}
\newcommand{\idl}{\mbox{Ideal}}
\newcommand{\qmod}{\mbox{Qmod}}
\newcommand{\ddd}{,\ldots,}

\DeclareMathOperator{\Rank}{rank}

\newcommand{\re}{\mathbb{R}}

\newcommand{\N}{\mathbb{N}}

\def\af{\alpha}

\newcommand{\st}{\mathit{s.t.}}
\newcommand{\reff}[1]{(\ref{#1})}

\newcommand{\lmd}{\lambda}
\newcommand{\pt}{\partial}
\newcommand{\dt}{\delta}
\newcommand{\nn}{\nonumber}

\newcommand{\mc}[1]{\mathcal{#1}}

\def\rank{\mbox{rank}}

\newcommand{\bdes}{\begin{description}}
\newcommand{\edes}{\end{description}}

\newcommand{\bal}{\begin{align}}
\newcommand{\eal}{\end{align}}

\newcommand{\bnum}{\begin{enumerate}}
\newcommand{\enum}{\end{enumerate}}

\newcommand{\bit}{\begin{itemize}}
\newcommand{\eit}{\end{itemize}}

\newcommand{\bea}{\begin{eqnarray}}
\newcommand{\eea}{\end{eqnarray}}
\newcommand{\be}{\begin{equation}}
\newcommand{\ee}{\end{equation}}

\newcommand{\baray}{\begin{array}}
\newcommand{\earay}{\end{array}}

\newcommand{\bsry}{\begin{subarray}}
\newcommand{\esry}{\end{subarray}}

\newcommand{\bca}{\begin{cases}}
\newcommand{\eca}{\end{cases}}

\newcommand{\bcen}{\begin{center}}
\newcommand{\ecen}{\end{center}}

\newcommand{\bbm}{\begin{bmatrix}}
\newcommand{\ebm}{\end{bmatrix}}

\newcommand{\btab}{\begin{tabular}}
\newcommand{\etab}{\end{tabular}}

\theoremstyle{plain}
\newtheorem{definition}{Definition}[section]

\newtheorem{theorem}[definition]{Theorem}
\newtheorem{assumption}[definition]{Assumption}
\newtheorem{example}[definition]{Example}
\newtheorem{alg}[definition]{Algorithm}

\theoremstyle{remark}

\numberwithin{equation}{section}

\begin{document}

\title[Rational GNEPs]
{Rational Generalized Nash Equilibrium Problems}

\author[Jiawang Nie]{Jiawang~Nie}
\author[Xindong Tang]{Xindong~Tang}
\author[Suhan Zhong]{Suhan~Zhong}
\address{Jiawang Nie, Department of Mathematics,
University of California San Diego,
9500 Gilman Drive, La Jolla, CA, USA, 92093.}
\email{njw@math.ucsd.edu}

\address{Xindong Tang, Department of Applied Mathematics,
The Hong Kong Polytechnic University,
Hung Hom, Kowloon, Hong Kong.
}
\email{xindong.tang@polyu.edu.hk}

\address{Suhan Zhong, Department of Mathematics,
Texas A\&M University, College Station, TX, USA, 77843.}
\email{suzhong@tamu.edu}

\date{}

\begin{abstract}
This paper studies generalized Nash equilibrium problems
that are given by rational functions.
The optimization problems are not assumed to be convex.
Rational expressions for Lagrange multipliers
and feasible extensions of KKT points
are introduced to compute a generalized Nash equilibrium (GNE).
We give a hierarchy of rational optimization problems to
solve rational generalized Nash equilibrium problems.
The existence and computation of feasible extensions are studied.
The Moment-SOS relaxations are applied to solve
the rational optimization problems.
Under some general assumptions, we show that the proposed hierarchy
can compute a GNE if it exists or detect its nonexistence.
Numerical experiments are given to show the efficiency of the proposed method.
\end{abstract}

\maketitle

\section{Introduction}

The generalized Nash equilibrium problem (GNEP) is a kind of game to find strategies
for a group of players such that each player's objective cannot be further optimized,
for given strategies of other players.
Suppose there are $N$ players and the $i$th player's strategy
is the real vector $x_i \in \re^{n_i}$. We write that
\[
x_i \coloneqq (x_{i,1},\ldots,x_{i,n_i}),\quad
x  \coloneqq  (x_1,\ldots,x_N).
\]
Let $n  \coloneqq  n_1+ {\cdots} + n_N$.
When the $i$th player's strategy $x_i$ is focused,
we also write that $x=(x_i,x_{-i})$, where
\[
x_{-i}\, \coloneqq \, (x_1,\ldots,x_{i-1},x_{i+1},\ldots,x_N) .
\]
A strategy tuple $u \coloneqq (u_1, \ldots, u_N)$
is said to be a generalized Nash equilibrium (GNE) if
each $u_i$ is the optimizer for the $i$th player's optimization
\be
\label{eq:GNEP}
\mbox{F}_i(u_{-i}): \,
\left\{ \begin{array}{cl}
\min\limits_{\xpi\in \Rni}  &  \fpi(x_i,u_{-i}) \\
\st & x_i\in X_i(\umi).
\end{array} \right.
\ee
In the above, the $X_i(\umi)$ is the feasible set
and $\fpi(x_i,u_{-i})$ is the $i$th player's objective.
They are parameterized by
$u_{-i} =(u_1,\ldots,u_{i-1},u_{i+1},\ldots,u_N)$.
Each player's optimization is parameterized by the strategies of other players.
We denote by $\mathcal{S}$ the set of all GNEs and denote by
$\mathcal{S}_i(u_{-i})$ the set of minimizers
for the optimization $\mbox{F}_i(u_{-i})$.
The entire feasible strategy set is
\be  \label{df:setX}
X \coloneqq
\left\{  (x_1,\ldots,x_N)
\left| \,
x_i \in X_i(\xmi)
\right. , i = 1, \ldots, N
\right \} .
\ee
A strategy tuple $x = (x_1,\ldots,x_N)$
is said to be feasible if each $x_i \in X_i(\xmi)$.

This paper studies {\it rational generalized Nash equilibrium problems}
(rGNEPs), i.e., all the objectives
and constraining functions are rational functions in $x$.
We assume the $i$th player's feasible set is given as
\be \label{eq:feaset}
X_i(\xmi) \, = \,
\left\{x_i \in \Rni \left|\begin{array}{l}
g_{i,j}(\xpi,\xmi) = 0 \, (j\in \mc{I}^{(i)}_0),\\
g_{i,j}(\xpi,\xmi) \geq 0 \, (j\in \mc{I}^{(i)}_1), \\
g_{i,j}(\xpi,\xmi) > 0 \, (j\in \mc{I}^{(i)}_2)
\end{array}\right.\right\},
\ee
where $\mc{I}^{(i)}_0,\mc{I}^{(i)}_1,\mc{I}^{(i)}_2$ are respectively the
labeling sets (possibly empty) for equality, weak inequality
and strict inequality constraints.
For the rational function to be well defined,
we assume all denominators are positive in the feasible set.
If this is not the case, we can add
strict inequality constraints for denominators.
Rational functions frequently appear in GNEPs.
When defining functions are polynomials,
the GNEPs are studied in the recent work
\cite{Nie2020gs,Nie2020nash,Nie2021convex}.
For convenience, rational functions are also called
rational polynomials throughout the paper.

A special case of GNEPs is the \textit{Nash equilibrium problems} (NEPs):
each feasible set $X_i(x_{-i})$ is independent of $x_{-i}$.
When NEPs are defined by polynomials,
a method is given in \cite{Nie2020nash} to solve them.
For GNEPs given by convex polynomials,
it is studied how to solve them in the recent work \cite{Nie2021convex}.
We refer to
\cite{dreves2011solution,facchinei2009generalized,
Facchinei2010,Facchinei2010book,Heusinger2012} for related work.

One may reformulate rGNEPs equivalently as polynomial GNEPs
by introducing new variables or changing the description of the feasible set.
However, doing so may lose some useful properties.
For instance, the convexity may be lost if
we use polynomial reformulations.
The following is such an example.

\begin{example}\rm
\label{ex:rat2poly}
Consider the $2$-player rGNEP
\be \label{eq:firsteprtn}
\begin{array}{cllcl}
    \min\limits_{x_{1} \in \re^2 }& \frac{2(x_{1,1})^2+(x_{1,2})^2+x_{1,1}x_{1,2}\cdot e^Tx_2}{x_{1,1}}&\vline&
    \min\limits_{x_{2} \in \re^2 }& \frac{2(x_{2,1})^2+(x_{2,2})^2-x_{2,1}x_{2,2} \cdot e^Tx_1}{x_{2,1}} \\
    \st & x_{1,1}-\frac{x_{2,1}}{x_{1,2}}\ge0 , &\vline&
     \st &1-e^T(x_2-x_1)\ge 0, \\
    &x_{1,1} > 0,\, x_{1,2} > 0, &\vline& & x_{2,1}-1 \ge 0, \, x_{2,2} - 1 \ge 0 .
\end{array}
\ee
In the above, $e =[ 1 \,\, 1]^{T}$.
In the domain $(x_1,x_2)>0$,
each player's optimization is convex in its strategy variable.
We can equivalently express this GNEP as polynomial optimization
\be\label{eq:polyreform}
\begin{array}{cllcl}
\min\limits_{x_{1} \in \re^3}& \substack{ x_{1,3}
   [2(x_{1,1})^2+(x_{1,2})^2+x_{1,1}x_{1,2}\cdot \hat{e}^T x_2]}     &\vline&
\min\limits_{x_{2} \in \re^3}& \substack{x_{2,3}
     [2(x_{2,1})^2+(x_{2,2})^2-x_{2,1}x_{2,2}\cdot \hat{e}^T x_1]} \\
\st & x_{1,1}x_{1,2}-x_{2,1}\ge0 , &\vline&
 \st &1-\hat{e}^T(x_2-x_1)\ge 0, \\
&x_{1,1} > 0,\, x_{1,2} > 0, &\vline& & x_{2,1}-1 \ge 0, \, x_{2,2} - 1 \ge 0 ,\\
& x_{1,1}x_{1,3} = 1, &\vline& & x_{2,1}x_{2,3} = 1,
\end{array}
\ee
where $\hat{e}=[ 1 \,\, 1 \,\, 0]^{T}$.
However, the above two optimization problems are not convex.
\end{example}

The GNEPs were originally introduced to model economic problems.
They are now widely used in various fields,
such as transportation, telecommunications, and machine learning.
We refer to \cite{ardagna2017generalized,chen2020oil,
Contreras2004,Kesselman2005,Liu2016,Pang2008}
for recent applications of GNEPs.
It is typically difficult to solve GNEPs.
The major challenge is due to interactions
among different players' strategies on the objectives and feasible sets.
The set of GNEs may be nonconvex,
even for convex NEPs (see \cite{Nie2020nash}).
Convex GNEPs can be reformulated as variational inequality (VI)
or quasi-variational inequality (QVI) problems
\cite{Facchinei2010generalized,nabetani2011parametrized,Pang2005}.
A semidefinite relaxation method for convex GNEPs
of polynomials is given in \cite{Nie2021convex}.
The penalty functions are used to solve GNEPs in \cite{Ba2020,FacKan10}.
An Augmented-Lagrangian method is given in \cite{kanzow2016}.
The Nikaido-Isoda function related methods are given in
\cite{dreves2012nonsmooth,vonHeusinger2009-2}.
Newton type methods are given in
\cite{facchinei2009generalized,Heusinger2012}.
An interior point method is given in \cite{dreves2011solution}.
Gauss-Seidel type methods are studied in \cite{Facchinei2011,Nie2020gs}.
Lemke's method is used to solve affine GNEPs \cite{Schiro2013}.
An ADMM-type method for solving GNEPs in Hilbert spaces is given in \cite{Borgens2021}.
Moreover, quasi-NEs for nonconvex GNEPs are studied in \cite{Cui2021book,Pang2011}.
We refer to \cite{Facchinei2010,Facchinei2010book,fischer2014generalized}
for surveys on GNEPs.

\subsection*{Contributions}

We study generalized Nash equilibrium problems
that are given by rational functions.
This is motivated by earlier work on polynomial NEPs \cite{Nie2020nash}
and convex GNEPs \cite{Nie2021convex}.
In various applications, people often face GNEPs given by rational functions.
Even for polynomial GNEPs, the Lagrange multiplier expressions
are usually given by rational functions instead of polynomial ones.
This was observed in \cite{Nie2021convex}.
Mathematically, rGNEPs can be equivalently formulated as polynomial GNEPs
by introducing new variables. However, such a reformulation usually destroys
some nice properties (e.g., convexity may be lost; see in Example~\ref{ex:rat2poly}).
Moreover, solving the reformulated polynomial GNEPs
is usually more computationally expensive.
This can be observed in numerical experiments.

For convex GNEPs, each feasible KKT point is a GNE. For nonconvex GNEPs,
a KKT point is typically not a GNE (see Example~\ref{eq:simpsimplex}).
When we solve nonconvex GNEPs,
the earlier existing methods may not get a GNE,
or are not able to detect its nonexistence.
There exists relatively little work for solving nonconvex GNEPs.
In this paper, we propose a new approach for solving rGNEPs.
The optimization problems are not assumed to be convex.
Our new approach is based on a hierarchy of rational optimization problems.
Our major contributions are:

\begin{itemize}

\item First, we introduce rational expressions for Lagrange multipliers
of each player's optimization.
These expressions can be used to give new constraints for GNEs.

\item Second, we introduce the new concept of
feasible extensions for some KKT points.
More specifically, for a KKT point that is not a GNE,
we extend it to the image of a rational function,
such that the image is feasible on the KKT set.
The feasible extension can be used to preclude KKT points that are not GNEs.
For nonconvex rGNEPs, the usage of rational feasible extensions
is important for computing a GNE (if it exists)
or for detecting its nonexistence.

\item Third, the Moment-SOS relaxations are used
to solve rational optimization problems that are obtained from
using Lagrange multiplier expressions and feasible extensions of some KKT points.
Unlike polynomial optimization, a rational optimization problem
may have strict inequalities.
We study the properties of Moment-SOS relaxations for solving them.

\end{itemize}

The paper is organized as follows.
Some preliminaries for moment and polynomial optimization
are given in Section~\ref{sc:pre}.
A hierarchy of rational optimization problems for solving
the GNEP is proposed in Section~\ref{sc:alg}.
Feasible extensions of KKT points are studied in Section~\ref{sc:RP}.
We show how to solve rational optimization problems in Section~\ref{sc:ROP}.
Some numerical experiments are given in Section~\ref{sc:ne}.
Some conclusions and discussions are given in Section~\ref{sc:conc}.

\section{Preliminaries}
\label{sc:pre}

\noindent {\bf Notation}
The symbol $\mathbb N$ denotes the set of
nonnegative integers. The symbol $\mathbb R$
denotes the set of real numbers.
For a positive integer $k$, denote the set $[k]  \coloneqq  \{1, \ldots, k\}$.
For a real number $t$, $\lceil t \rceil$
denotes the smallest integer not smaller than $t$.
We use $e_i$ to denote the vector such that the $i$th entry is
$1$ and all others are zeros, use $e$
to denote the vector of all ones.
For a vector $u$ in the Euclidean space,
its Euclidean norm is denoted as $\| u \|$.
By writing $A\succeq0$ (resp., $A\succ0$), we mean that the matrix $A$
is symmetric positive semidefinite (resp., positive definite).
Let $\nfR[x]$ denote the ring of real polynomials in $x$ and
$\nfR[x]_d$ denotes the set of polynomials with degrees not bigger than $d$.
For the $i$th player's strategy vector $x_i$,
the notation $\re[x_i]$ and $\re[x_i]_d$
are defined similarly. For a polynomial $p\in\mathbb{R}[x]$,
we write $p=0$ to mean that $p$ is the identically zero polynomial,
and $p \ne 0$  means that $p$ is not identically zero.
The total degree of $p$ is denoted by $\deg(p)$
and its partial degree on $x_i$ is denoted by $\deg_{x_i}(p)$.
For a function $f(x)$, the notation
$\nabla_{x_i}f:=(\frac{\partial f}{\partial x_{i,j}})_{j\in[n_i]}$
denotes its gradient with respect to $x_i$.
For a set $X$, we use $cl(X)$ to denote its closure in the Euclidean topology.
A property is said to hold {\it generically}
if it holds for all points in the space of input data
except a set of Lebesgue measure zero.

Let $z=(z_1,\ldots, z_l)$ stand for the vector $x$ or $x_i$.
For a power $\af  \coloneqq  (\af_1, \ldots, \af_l) \in \N^{l}$, we denote that
$z^\alpha  \coloneqq  z_1^{\alpha_1} \cdots z_l^{\alpha_l}$
and $|\alpha| \coloneqq \alpha_1+{\cdots}+\alpha_l.$
For a degree $d >0$, denote the power set
$
\N_d^l  \coloneqq
\{\alpha\in {\mathbb{N}}^l:  \ |\alpha| \le d \}.
$
We use $[z]_d$ to denote the vector of all monomials in $z$
whose degrees are at most $d$, ordered in the graded alphabetical ordering, i.e.,
$
[z]_d \coloneqq [1,\,  z_1,\, \ldots,\, z_l,\, z_1^2,\,\ldots,\, z_l^d ]^T.
$

\subsection{Ideals and quadratic modules}
\label{ssc:poly}

For a polynomial $p \in\mathbb{R}[x]$ and subsets $I, J \subseteq \mathbb{R}[x]$,
define the product and Minkowski sum
\[
p \cdot I   \coloneqq \{ p  q: \, q \in I \}, \quad
I+J   \coloneqq  \{a+b: \, a \in I, b \in J  \}.
\]
The subset $I$ is an ideal if $p \cdot I \subseteq I$
for all $p\in\mathbb{R}[x]$ and $I+I \subseteq I$.
The ideal generated by a polynomial tuple $h = (h_1, \ldots, h_{m_1})$ is
$
\idl[h] \coloneqq  h_1\cdot\mathbb{R}[x] + {\cdots} + h_{m_1} \cdot \mathbb{R}[x].
$
For a degree $d$, the $d$th truncation of $\idl[h]$ is
\[
\idl[h]_{d} \coloneqq h_1\cdot\mathbb{R}[x]_{d-\deg(h_1)}+\cdots+
h_{m_1} \cdot\mathbb{R}[x]_{d-\deg(h_{m_1})}.
\]

A polynomial $\sigma \in \re[x]$ is said to be a sum-of-squares (SOS)
if $\sigma = p_1^2+{\cdots}+p_k^2$ for some $p_i \in\nfR[x]$.
We use $\Sigma[x]$ to denote the set of all SOS polynomials in $x$
and denote the truncation
$
\Sigma[x]_d  \coloneqq  \Sigma[x] \cap \nfR[x]_d.
$
The quadratic module of a polynomial tuple $g=(g_1,\ldots,g_{m_2})$ is
$
\qmod[g]  \coloneqq  \Sigma[x] +  g_1 \cdot \Sigma[x] +
{\cdots} + g_{m_2} \cdot  \Sigma[x].
$
Similarly, the degree-$d$ truncation of $\qmod[g]$ is
\[
\qmod[g]_{d} \,  \coloneqq  \, \Sigma[x]_{d} + g_1\cdot \Sigma[x]_{d-\deg(g_1)}
+{\cdots}+g_{m_2}\cdot\Sigma[x]_{d-\deg(g_{m_2})}.
\]
The polynomial tuples $h,g$ determine the basic closed semi-algebraic set
\begin{equation}
  \label{polyrep}
T\,  \coloneqq  \,  \{x \in \nfR^n:
h_i(x) =  0 \, (i\in [m_1]),
g_j(x) \ge  0 \, (j \in [m_2])
\}.
\end{equation}
Clearly, every polynomial in $\idl[h] + \qmod[g]$
is nonnegative on the set $T$. We denote by
$\mathscr{P}(T)$ the set of polynomials nonnegative on $T$
and denote the  truncation
$\mathscr{P}_d(T)  \coloneqq \mathscr{P}(T)\cap\mathbb{R}[x]_d$.
Clearly, $\idl[h]+ \qmod[g] \subseteq\mathscr{P}(T)$.
The sets $\mathscr{P}(T)$, $\mathscr{P}_d(T)$ are convex cones,
and $\mathscr{P}_d(T)$ is the dual cone of the moment cone
\[
\mathscr{R}_d(T)  \coloneqq  \left\{
\sum_{i=1}^M \lmd_i [u_i]_{d}:\,  u_i \in T, \lmd_i \ge 0, M \in \N
\right\} .
\]
When $T$ is compact, the cone $\mathscr{R}_d(T)$
is closed and it equals the dual cone of $\mathscr{P}_d(T)$.

The set $\idl[h]+\qmod[g]$ is said to be {\it archimedean}
if there exists $p \in \idl[h]+\qmod[g]$ such that the inequality
$p(x) \ge 0$ defines a compact set.
If $\idl[h]+\qmod[g]$ is archimedean, then $T$ is compact.
Conversely, if $T$ is compact, say,
$T$ is contained in the ball $\|x\|^2 \le R$,
then $\idl[h]+\qmod[g,R -\|z\|^2]$ is archimedean.
When $\idl[h]+\qmod[g]$ is archimedean,
if a polynomial $p > 0$ on $T$, then $p \in \idl[h]+\qmod[g]$.
This conclusion is referenced as
Putinar's Positivstellensatz \cite{putinar1993positive}.

\subsection{Localizing and moment matrices}
\label{ssc:locmat}

For an integer $k\ge 0$, a real vector
$y=(y_{\alpha})_{\alpha\in\mathbb{N}_{2k}^n}$ is said to be a
{\it truncated multi-sequence} (tms) of degree $2k$.
For a polynomial $f = \sum_{ \af \in \N^n_{2k} } f_\af x^\af$,
define the operation
\be \label{<f,y>}
\langle f, y \rangle \,  \coloneqq  \,
\sum_{ \af \in \N^n_{2k} } f_\af y_\af.
\ee
The operation $\langle f, y \rangle$ is bilinear in $f$ and $y$.
For a polynomial $q \in \re[x]_{2t}$ ($t \le k$) and a degree
$s\le k - \lceil \deg(q)/2 \rceil$,
the $k$th order {\it localizing matrix} of $q$ for $y$
is the symmetric matrix $L_{q}^{(k)}[y]$ such that
(the $vec(a)$ denotes the coefficient vector of $a$)
\be \label{df:Lf[y]}
\langle qa^2, y \rangle  \, =  \,
vec(a)^T \big( L_{q}^{(k)}[y] \big) vec(a)
\ee
for all $a \in \re[x]_s$.
When $q=1$ (the constant one polynomial),
the localizing matrix $L_{q}^{(k)}[y]$
becomes the $k$th order \textit{moment matrix}
$M_k[y]  \coloneqq  L_{1}^{(k)}[y]$.

Localizing and moment matrices can be used to approximate
the moment cone $\mathscr{R}_d(T)$ by semidefinite programming relaxations.
They are useful for solving polynomial, matrix and tensor optimization
\cite{HilNie08,PMI2011,Nie2012,NieZhang18}.
We refer to \cite{Las01,Las15,Lau09}
for a general introduction to polynomial optimization and moment problems.

\subsection{Lagrange multiplier expressions}
\label{sc:LME}

The Karush-Kuhn-Tucker (KKT)
conditions are useful for solving GNEPs and NEPs.
We review optimality conditions for nonlinear optimization (see \cite{Brks}).
Frequently used constraint qualifications are
the linear independence constraint qualification (LICQ) and
the Mangasarian-Fromovite constraint qualification (MFCQ).
For strict inequality constraints,
their associated Lagrange multipliers are zeros,
and hence the KKT conditions only concern weak inequality constraints.
For the convenience of description, we write that
$\mc{I}^{(i)}_0 \cup  \mc{I}^{(i)}_1 = \{1, \ldots, m_i \}$
and $g_i = (g_{i,1},\ldots,g_{i,m_i})$.
Under certain constraint qualifications,
if $x_i\in X_i(\xmi)$ is a minimizer of $\mbox{F}_i(\xmi)$,
then there exists a Lagrange multiplier vector
$\lambda_i \coloneqq (\lambda_{i,1},\ldots,\lambda_{i,m_i})$ such that
\be
\label{eq:KKTwithLM}
\left\{
\begin{array}{l}
\nabla_{x_i} f_i(x)-\sum_{j=1}^{m_i}\lambda_{i,j}
       \nabla_{x_i} g_{i,j}(x)=0, \\
   \lambda_i\perp g_i(x), \,
   \lambda_{i,j}\ge0 \, (j\in\mc{I}^{(i)}_1).
\end{array}
\right.
\ee
In the above, $\lambda_i \perp g_i(x)$ means that
$\lmd_i$ is perpendicular to $g_i(x)$.
The system (\ref{eq:KKTwithLM}) gives the first order
KKT conditions for $\mbox{F}_i(\xmi)$.
Such $(x_i,\lambda_i)$ is called a critical pair.
Under the constraint qualifications, every GNE
satisfies (\ref{eq:KKTwithLM}).

Consider the $i$th player's optimization problem $\mbox{F}_i(\xmi)$.
If there exists a rational vector function $\tau_i(x)$
such that $\lmd_i = \tau_i(x)$ for every critical pair
$(x_i,\lambda_i)$ of $\mbox{F}_i(\xmi)$,
then $\tau_i(x)$ is called a
\textit{rational Lagrange multiplier expression} (LME) for $\lmd_i$.
As in (\ref{eq:KKTwithLM}),
each critical pair $(x_i,\lambda_i)$ of the optimization $\mbox{F}_i(\xmi)$ satisfies
\be
\label{eq:Clmd=df}
\underbrace{\bbm
\nabla_{x_i} g_{i,1}(x) & \nabla_{x_i} g_{i,2}(x) &  \cdots &  \nabla_{x_i} g_{i,m_i}(x) \\
g_{i,1}(x) & 0  & \cdots & 0 \\
0  & g_{i,2}(x)  & \cdots & 0 \\
\vdots & \vdots & \ddots & \vdots \\
0  &  0  & \cdots & g_{i,m_i}(x)
\ebm}_{G_i(x) }
\underbrace{
\bbm  \lmd_{i,1} \\ \lmd_{i,2} \\ \vdots \\ \lmd_{i,m_i} \ebm
}_{\lmd_i}
=
\underbrace{\bbm  \nabla_{x_i}f_i(x)  \\ 0 \\ \vdots \\ 0 \ebm}_{ \hat{f_i}(x)} .
\ee
If there exist a matrix polynomial $T_i(x)$
and a nonzero scalar polynomial $q_i(x)$ such that
\[
T_i(x)G_i(x) \, = \, q_i(x)  I_{m_i},
\]
then (\ref{eq:Clmd=df}) implies that
$  q_i(x)\lambda_i=T_i(x)\hat{f}_i(x).$
This gives the rational LME:
\be
\label{eq:rtnlmd_i}
\tau_i(x) \, = \, T_i(x)\hat{f}_i(x) / q_i(x).
\ee
At a point $u$, if $q_i(u)=0$, then $T_i(u)\hat{f}_i(u)=0$.

The rational expression (\ref{eq:rtnlmd_i}) almost always exists.
This can be shown as follows. Let $H_i(x) \coloneqq G_i(x)^TG_i(x)$,
then $H_i(x)$ is a matrix of rational functions
and $H_i(x) \succeq 0$ on $X$.
If the determinant $\det H_i(x)$ is not identically zero (this is the general case),
then we have
\[
\mbox{adj}\, H_i(x) \cdot  H_i(x) \,\, = \,\, \det H_i(x) \cdot  I_{m_i},
\]
where $\mbox{adj}\, H_i(x)$ denotes the adjacent matrix of $H_i(x)$.
Let $d_i(x)$ be the denominator of $\det H_i(x)$, then
$T_i(x)G_i(x) =  q_i(x)\cdot I_{m_i}$
for the selection
\begin{equation}
\label{eq:genLME}
T_i(x) = d_i(x) \cdot \mbox{adj} \, H_i(x) \cdot  G_i(x)^T,\quad
q_i(x) = d_i(x)  \cdot \det H_i(x).
\end{equation}
The above choices of $T_i(x)$ and $q_i(x)$ may not be computationally efficient.
However, there often exist different options for
$T_i(x)$ and $q_i(x)$ to make \reff{eq:rtnlmd_i} hold.
For computational efficiency, we prefer that
$T_i(x)$ and $q_i(x)$ have low degrees.
It is worth noting that once their degrees are given,
the equation $T_i(x)G_i(x) = q_i(x)\cdot I_{m_i}$
is linear in the coefficients of $T_i(x)$ and $q_i(x)$.
So we can obtain $T_i(x), q_i(x)$ by solving linear equations.
The following is such an example.

\begin{example}\rm
Let $x = (x_1,x_2), x_1 \in \re^1, x_2 \in \re^1$ and $g_2(x) = (1-x_1-x_2,x_2)$.
We look for $T_2(x),\,q_2(x)$ such that $T_2(x)G_2(x) = q_2(x)\cdot I_2$, where
\[
G_2(x) = \left[\begin{array}{rr}
-1 & 1 \\1-x_1-x_2 & 0 \\0 & x_2
\end{array}\right].
\]
We consider $q_2(x)$ and $T_2(x)$ having degree $1$, i.e.,
\[
\begin{array}{l}
T_2(x) = (a_{i,j}+b_{i,j}x_1+c_{i,j}x_2)_{1\le i\le 2,1\le j\le 3},\\
q_2(x) = a_0+b_0x_1+c_0x_2.
\end{array}
\]
The equality $T_2(x)G_2(x) = q_2(x)\cdot I_2$ gives the equations
\[
\begin{array}{l}
a_{1,1} = b_{1,1} = b_{1,2} = c_{1,2} = b_{2,2} = c_{2,2} = b_{1,3} =
c_{1,3} = b_{2,3} = c_{2,3} = 0,\\
a_{0} = a_{2,1} = a_{1,2} = a_{2,2} = -b_{2,1} = -c_{2,1} = -b_0,\\
a_{1,3} = -c_{1,1},\, c_0 = -c_{1,1}-a_{1,2},\, a_{2,3} = c_0-c_{2,1}.
\end{array}
\]
We can choose $a_0 = 1$ and $c_{1,1} = -1$ to obtain
\[
T_2(x) = \left[\begin{array}{rrr}
-x_2 & 1 & 1\\ 1-x_1-x_2 & 1 & 1
\end{array}\right],\quad
q_2(x) = 1-x_1.
\]
We refer to \cite{Nie2019,Nie2021convex} for more details
about Lagrange multiplier expressions.
\end{example}

\section{A hierarchy of optimization problems}
\label{sc:alg}

In this section, we propose a new approach for solving rGNEPs.
It requires solving a hierarchy of rational optimization problems.
They are obtained from Lagrange multiplier expressions
and feasible extensions of KKT points that are not GNEs.
Under some general assumptions, we prove that this hierarchy
either returns a GNE or detects its nonexistence.

As shown in Subsection~\ref{sc:LME},
one can express Lagrange multipliers as rational functions on the KKT set.
Recall the set $X$ as in \reff{df:setX}.
For the $i$th player's optimization $\mbox{F}_i(x_{-i})$,
we suppose that there is a tuple
$\tau_i=(\tau_{i,j})_{j \in \mc{I}^{(i)}_0 \cup \mc{I}^{(i)}_1  }$
of rational functions in $x$,
with denominators positive on $X$, such that
\begin{equation}
\label{eq:as:ratexp}
\lambda_{i,j} \, = \,  \tau_{i,j} (x),
\quad  j \in \mc{I}^{(i)}_0 \cup \mc{I}^{(i)}_1 ,
\end{equation}
for each critical pair $(x_i,\lambda_i)$ of $\mbox{F}_i(x_{-i})$.
When $G_i(x)$ has full column rank on $X$,
there exist LMEs satisfying \reff{eq:as:ratexp},
by \cite[Proposition~3.6]{Nie2021convex}.
Note that the Lagrange multipliers are zero for strict inequality constraints.
So, the KKT set is
\be \label{eq:KKTweak}
\mathcal{K}  \coloneqq
\left\{
 x\in X\left|
\begin{array}{ccc}
\nabla_{x_i}f_i= \sum\limits_{j \in  \mc{I}^{(i)}_0 \cup  \mc{I}^{(i)}_1  }
    \tau_{i,j}(x)\nabla_{x_i} g_{i,j}(x), \, (i \in [N])\\
\tau_{i,j}(x)  g_{i, j}(x) = 0, \,
\tau_{i,j}(x)\ge 0, \,(i \in [N],\,j\in\mc{I}^{(i)}_1 )
\end{array}
\right.
 \right\}.
\ee

Not every point $u = (u_1, \ldots, u_N) \in \mc{K}$ is a GNE.
How do we preclude non-GNEs in $\mc{K}$?
We consider the case that $u$ is not a GNE. Then there exist
$i \in [N]$ and a point $v_i \in X_i(u_{-i})$ such that
\be \label{ineq:f(v_i,u)}
f_i(v_i,u_{-i})-f_i(u_i, u_{-i}) \,  < \,  0.
\ee
However, if $x \coloneqq (x_1, \ldots, x_N)$ is a GNE
and $v_i$ is also feasible for $\mbox{F}_i(x_{-i})$,
i.e., $v_i \in X_i(x_{-i})$,
then $x$ must satisfy the inequality
\be  \label{ineq:fv-fx}
f_i(v_i,x_{-i}) - f_i(x_i, x_{-i}) \ge 0.
\ee
That is, every GNE $x$ satisfies the constraint
\reff{ineq:fv-fx} if $v_i \in X_i(x_{-i})$.
This is used to solve NEPs in \cite{Nie2020nash}.
However, unlike NEPs, the feasible set of $X_i(x_{-i})$ depends on $x_{-i}$.
As a result, a point $v_i \in X_i(u_{-i})$ may not be feasible for $\mbox{F}_i(x_{-i})$,
i.e., it is possible that $v_i\not\in X_i(x_{-i})$ for a GNE $x$.
For such a case, the inequality (\ref{ineq:fv-fx}) may not hold for any GNEs.
In other words, it is possible that for every GNE $x^* = (x_i^*,x_{-i}^*)$,
it may happen that $v_i\not\in X_i(x^*_{-i})$ and
\[
f_i(v_i,x^*_{-i}) <  f_i(x_i^*, x_{-i}^*) =
\min\limits_{x_i\in X(x_{-i}^*)} f_i(x_i,x^*_{-i}) .
\]
The following is such an example.

\begin{example}\rm
\label{eq:simpsimplex}
Consider the $2$-player GNEP
\[
\begin{array}{cllcl}
    \min\limits_{x_{1} \in \re^2 }& (x_{1,1}-x_{1,2})x_{2,1}x_{2,2}-x_1^Tx_1 &\vline&
    \min\limits_{x_{2} \in \re^2 }& 3(x_{2,1}-x_{1,1})^2+2(x_{2,2}-x_{1,2})^2 \\
    \st & 1-e^Tx\ge 0,\,x_1\ge0 ,&\vline& \st & 2-e^Tx\ge 0,\,x_2\ge0.\\
\end{array}
\]
It has only two GNEs $x^*=(x_1^*,\,x_2^*)$:
\[
x_1^*=x_2^*=(0.5,0)\quad \text{and}\quad x_1^*=x_2^*=(0,0.5).
\]
Consider the point $u=(u_1,u_2) \in \mc{K}$, with $u_1=u_2=(0,0)$.
The $u_1$ is not a minimizer of $\mbox{F}_1(u_2)$,
so $u$ is not a GNE.
The optimizers of $\mbox{F}_1(u_2)$ are
$v_1=(1,0)$ and $(0,1)$.
One can check that for either GNE $x^*$, it holds that
\[
v_1 \not\in X_1( x_2^* ),  \qquad
f_1(v_1,x_2^*)-f_1(x^*_1, x^*_2)=-0.75<0.
\]
The inequality \reff{ineq:fv-fx} does not hold for any GNE.
\end{example}

The above example shows that the constraint \reff{ineq:fv-fx}
may not hold for any GNE.
However, if there is a function $p_i$ in $x$ such that
\be \label{eq:ratext}
v_i=p_{i}(u), \qquad p_{i}(x)\in X_i(\xmi)\quad
 \mbox{for all} \quad  x\in \mc{K},
\ee
then the following inequality
\be \label{ineq:f(p)-f(x)}
f_i(p_i(x),x_{-i})-f_i(x_i, x_{-i}) \, \ge \, 0
\ee
separates GNEs and non-GNEs. This is because
$f_i(x_i, x_{-i}) \le  f_i(p_{i}(x),x_{-i})$
for every GNE $x$, since $p_i(x) \in X_{i}(x_{-i})$.
This motivates us to make the following assumption.

\begin{assumption}\label{as:ratext}
For a given triple $(u,i,v_i)$,  with $u \in \mc{K}$,
$i\in [N]$ and $v_i \in \mc{S}_i(u_{-i})$,
there exists a rational vector-valued function $p_i$ in
$x \coloneqq (x_1, \ldots, x_N)$ such that \reff{eq:ratext} holds.
\end{assumption}

The function $p_{i}$ satisfying \reff{eq:ratext}
is called a {\it feasible extension} of $v_i$ at the point $u$.
Feasible extension is useful for solving bilevel optimization~\cite{Nie2020bilevel}.
In Section~\ref{sc:RP}, we will discuss the existence
and computation of such $p_{i}$.

\subsection{An algorithm for solving GNEPs}

Based on LMEs and feasible extensions,
we propose the following algorithm for solving GNEPs.

\begin{alg} \label{ag:KKTSDP} \rm
For the given GNEP of (\ref{eq:GNEP}), do the following:

\begin{itemize}

\item [Step~0]
Find the Lagrange multiplier expressions as in \reff{eq:as:ratexp}.
Let $\mathscr{U} \coloneqq \mathcal{K}$ and $k \coloneqq 0$.
Choose a generic positive definite matrix $\Theta$ of length $n+1$.

\item [Step~1]
Solve the following optimization
(note $[x]_1 = \bbm 1 & x^T \ebm^T$)
\be \label{eq:KKTwithpolyext}
\left\{ \baray{rl}
\min & [x]_1^T\Theta[x]_1  \\
    \st & x\in \mathscr{U} .
\earay \right.
\ee
If (\ref{eq:KKTwithpolyext}) is infeasible, output that
either (\ref{eq:GNEP}) has no GNEs or
there is no GNE in the set $\mc{K}$.
Otherwise, solve it for a minimizer
$u \coloneqq (u_1, \ldots, u_N)$, if it exists.

\item [Step~2]
For each $i=1, \dots, N$, solve the following optimization
\begin{equation}
\label{eq:checkopt:alg}
\left\{
\baray{rl}
\delta_i \coloneqq  \min  & f_i(x_i,u_{-i})-f_i(u_i, u_{-i})\\
\st\quad & x_i\in X_i(u_{-i})
\earay
\right.
\end{equation}
for a minimizer $v_i$. Denote the label set
\begin{equation}
\label{eq:Delta}
\mathcal{N} \coloneqq \{i\in[N]: \delta_i<0\}.
\end{equation}
If $\mathcal{N}=\emptyset$, then $u$ is a GNE and stop;
otherwise, go to Step~3.

\item [Step~3]
For every above triple $(u, i, v_i)$ with $i\in \mathcal{N}$,
find a rational feasible extension $p_{i}$ satisfying \reff{eq:ratext}.
Then update the set $\mathscr{U}$ as
\be  \label{def:mathscrK}
\mathscr{U} \coloneqq \mathscr{U} \cap
\big\{x\in\re^n: f_i(p_i(x),x_{-i}) - f_i(x_i, x_{-i}) \ge 0 \,
\forall \,i\in\mathcal{N} \big \} .
\ee
Then, let $k \coloneqq k+1$ and go to Step~1.
\end{itemize}
\end{alg}

In Step~0, we can let $\Theta \coloneqq R^T R$
for a generically generated square matrix $R$.
Then the objective $[x]_1^T\Theta[x]_1$ is generic, coercive
and strictly convex,
and so, the optimization problem~\reff{eq:KKTwithpolyext}
has a unique minimizer if it is feasible.
This gives computational convenience for solving rational optimization
with Moment-SOS relaxations (see Theorem~\ref{thm:ratforall}).
Note that Algorithm~\ref{ag:KKTSDP} is applicable for all choices of
$\Theta$ (e.g., $\Theta = I_{n+1}$).
But a generically selected positive definite $\Theta$
is usually preferable in computational practice.
The optimization problem \reff{eq:KKTwithpolyext}
may have constraints given by rational polynomials
or it may have strict inequality constraints.
The optimization~\reff{eq:checkopt:alg}
may have both rational objective and rational constraints.
They can be solved by Moment-SOS relaxations.
The optimization problem~\reff{eq:checkopt:alg} has
a nonempty feasible set since $u_i\in X_i(u_{-i})$.
In applications, people usually assume (\ref{eq:checkopt:alg}) has a minimizer.
For instance, this is the case if its feasible set is compact
or if its objective is coercive.
We discuss how to solve the appearing rational
optimization problems in Section~\ref{sc:ROP}.

If a GNE is a KKT point, i.e.,
it belongs to the set $\mc{K}$ as in \reff{eq:KKTweak},
then it belongs to the set $\mathscr{U}$ in every loop.
In other words, the update of $\mathscr{U}$ in Algorithm~\ref{ag:KKTSDP}
does not preclude any GNEs.
The set $\mathscr{U}$ stays nonempty
if there is a GNE lying in $\mc{K}$.

In Algorithm~\ref{ag:KKTSDP}, we need LMEs and feasible extensions.
As shown in Subsection~\ref{sc:LME}, LMEs almost always exist.
For standard constraints like box, simplex or balls,
explicit LMEs are given in \reff{eq:boxLME}-\reff{eq:ballLME}.
When denominators of LMEs vanish at some points,
Algorithm~\ref{ag:KKTSDP} is still applicable,
because denominators can be cancelled by
multiplying their least common multiples.
We refer to Example~\ref{ep:jointsimp} for such cases.
The existence of a feasible extension is ensured if $\mc{K}$
is a finite set (see Theorem~\ref{thm:finitK}).
There exist explicit expressions for many common constraints;
see Subsection~\ref{ssc:commonFE}.
In summary, Algorithm~\ref{ag:KKTSDP}
can be used for solving many rGNEPs.

\subsection{Convergence analysis}
\label{ssc:cvg}

We now study the convergence of Algorithm~\ref{ag:KKTSDP}.

First, an interesting case is the convex rGNEP.
A GNEP is said to be convex if
every player's optimization problem is convex: for each fixed $x_{-i}$,
the objective $f_i(x_i,x_{-i})$ is convex in $x_i$,
the inequality constraining functions in \reff{eq:feaset} are concave in $x_i$
and all equality constraining functions are linear in $x_i$.
Interestingly, the concavity of constraining functions
can be weakened to the convexity of feasible sets
under certain assumptions.
As in \cite{Las10CovRep}, for given $x_{-i}$,
the feasible set $X_i(x_{-i})$ is said to be {\it nondegenerate}
if for every $j \in \mc{I}^{(i)}_0 \cup \mc{I}^{(i)}_1$,
the gradient $\nabla_{x_i} g_{i,j}(x)\ne 0$
for all $x_i \in X_i(x_{-i})$ such that $g_{i,j}(x)=0$.
The set $X_i(x_{-i})$ is said to satisfy {\it Slater's condition}
if it contains a point that makes all inequalities strictly hold.

\begin{theorem}
\label{tm:KKTSDPconvex}
Assume the Lagrange multipliers are expressed as in
\reff{eq:as:ratexp} with denominators positive on $X$.
Suppose that each objective $f_i$ is convex in $x_i$,
each $g_{i,j}$ is linear in $x_i$ for $j \in \mc{I}^{(i)}_0$,
and each strategy set $X_i(x_{-i})$ is convex, nondegenerate,
and satisfies Slater's condition.
Then, Algorithm~\ref{ag:KKTSDP} terminates at the initial loop $k=0$,
and it either returns a GNE or detects nonexistence of GNEs.
\end{theorem}
\begin{proof}
Under the given assumptions, a feasible point is a minimizer of
the optimization $\mbox{F}_i(x_{-i})$ if and only if it is a KKT point.
This is shown in \cite{Las10CovRep}.
Equivalently, a point is a GNE if and only if
it belongs to the set $\mc{K}$.
If there is a GNE, Algorithm~\ref{ag:KKTSDP}
can get one in Step~2 for the initial loop $k=0$,
and then it terminates.
If there is no GNE, the KKT point set $\mc{K}$ is empty,
then Algorithm~\ref{ag:KKTSDP} terminates in Step~1 for the initial loop.
\end{proof}

We remark that if there exist a matrix function $T_i(x)$ and a scalar function
$q_i(x)$ such that \[ T_i(x) G_i(x) = q_i(x) I_{m_i} \]
and $q_i(x) > 0$ on $X$ (see \reff{eq:Clmd=df} for $G_i(x)$),
then $X_i(\xmi)$ must be nondegenerate.
This can be implied by \cite[Proposition~3.6]{Nie2021convex}.
Moreover, when each $g_{i,j}$ is linear in $x_i$ for $j \in \mc{I}^{(i)}_0$
and every $g_{i,j}$ is concave in $x_i$ for $j \in \mc{I}^{(i)}_1$,
the $X_i(\xmi)$ is nondegenerate when it satisfies
Slater's condition \cite{Las10CovRep}.
When the nondegeneracy condition fails,
a GNE may not be a KKT point,
even under the convexity assumption and Slater's condition.
The following is such an example.

\begin{example}
\label{eq:infKKT}  \rm
Consider the GNEP
\begin{equation}
\label{eq:degenGNEP}
\begin{array}{cllcl}
\min\limits_{x_{1} \in \re^2 }& 2x_{1,1}+x_{1,2} &\vline&
\min\limits_{x_{2} \in \re^2 }& \|x_1+x_2\|^2 \\
\st & x_1^Tx_2\ge 0,\ x_{1,1}x_{1,2}\ge0,&\vline&
\st & x_{2,1}- 1\ge0,\ x_{2,2}-1\ge0 .
\end{array}
\end{equation}
In the above, all player's objectives and feasible sets are convex,
and Slater's condition holds.
The feasible set $X_1(x_2)$ is degenerate.
The KKT system for this GNEP is
\be\label{eq:degenKKT}
\left\{\begin{array}{c}
e+e_1 = x_2\lmd_{1,1}+(x_{1,1}e_2+x_{1,2}e_1)\lmd_{1,2},\\
2(x_1+x_2)=e_1\cdot\lmd_{2,1}+e_2\cdot \lmd_{2,2},\\
\lmd_{1,1}\cdot x_1^Tx_2=0,\ \lmd_{1,2}\cdot x_{1,1}x_{1,2}=0,\\
\lmd_{2,1}\cdot (x_{2,1}-1)=0,\ \lmd_{2,2}\cdot (x_{2,2}-1)=0,\\
x_1^Tx_2\ge0,\ x_{1,1}x_{1,2}\ge0,\ x_{2,1}\ge1,\ x_{2,2}\ge1,\\
\lmd_{1,1}\ge0,\ \lmd_{1,2}\ge0,\ \lmd_{2,1}\ge0,\ \lmd_{2,2}\ge0.
\end{array}
\right.
\ee
One may check that (\ref{eq:degenKKT}) has no solutions,
i.e., this convex GNEP does not have any KKT point.
However, the first player's feasible set is degenerate at $x_1 = (0,0)$,
which corresponds to the unique GNE
\[
x^* = (x_1^*,x_2^*),\quad x_1^* = (0,0),\  x_2^* = (1,1).
\]
Since the feasible set is degenerate,
there do not exist LMEs in the form of (\ref{eq:rtnlmd_i})
that have denominators positive on $X$.
However, if we choose
\be\label{eq:degenLME}
\left\{
\begin{array}{rcl}
T_1(x) &=&
\lvt\begin{array}{cccc}
-x_{1,1}x_{1,2} & 0 & x_{1,2} & x_{1,2}\\
x_1^Tx_2 & 0 & -x_{2,1} & -x_{2,1}
\end{array}
\rvt, \\
T_2(x) &=&
\lvt\begin{array}{cccc}
1 & 0 & 0 & 0\\
0 & 1 & 0 & 0\\
\end{array}
\rvt,   \\
q_1(x) &=& (x_{1,2})^2x_{2,2}, \\
q_2(x) &=&  1,
\end{array}
\right.
\ee
then $T_i(x)G_i(x) =  q_i(x)  I_{m_i}$ for each $i = 1, 2$,
and (\ref{eq:rtnlmd_i}) gives the LMEs:
\[
\lmd_{1,1}=\frac{-x_{1,1} }{x_{1,2} x_{2,2}}\frac{\pt f_1}{\pt x_{1,1}},  \quad \lmd_{1,2}=\frac{x_1^Tx_2}{(x_{1,2})^2x_{2,2}}\frac{\pt f_1}{\pt x_{1,1}},
\]
\[
\lmd_{2,1}= \frac{\pt f_2}{\pt x_{2,1}},  \quad
\lmd_{2,2}= \frac{\pt f_2}{\pt x_{2,2}} .
\]
The denominator $q_1$ has zeros on $X$. Interestingly,
Algorithm~\ref{ag:KKTSDP} still finds the GNE in the initial loop
(see Example~\ref{ep:firstepinsc4}(iv)).
\end{example}

Second, we prove that Algorithm~\ref{ag:KKTSDP}
terminates within finitely many loops under
a finiteness assumption on KKT points.
Recall that $\mc{S}$ denotes the set of all GNEs.
When the complement $\mc{K} \setminus \mc{S}$ is a finite set,
Algorithm~\ref{ag:KKTSDP} must terminate
within finitely many loops.

\begin{theorem}  \label{cvg:K-S:finite}
Assume the Lagrange multipliers are expressed as in \reff{eq:as:ratexp}.
Suppose Assumption~\ref{as:ratext} holds for every triple
$(u, i, v_i)$ produced by Algorithm~\ref{ag:KKTSDP}.
If the complement set $\mc{K}\setminus \mc{S}$ is finite,
then Algorithm~\ref{ag:KKTSDP} must terminate within finitely many loops,
and it either returns a GNE or detects its nonexistence.
\end{theorem}
\begin{proof}
When $\mathcal{K}\setminus\mathcal{S}=\emptyset$,
the algorithm terminates in the initial loop $k = 0$.
When $\mathcal{K}\setminus\mathcal{S}  \ne \emptyset$
and some $u\in\mathcal{K}\setminus\mathcal{S}$
is the minimizer of (\ref{eq:KKTwithpolyext}),
the set $\mc{N} \ne \emptyset$.
For each $i \in \mc{N}$, there exists
$v_i\in\mathcal{S}_i(u_{-i})$ such that
\[
\delta_i = f_i(v_i,u_{-i})-f(u_i, u_{-i})<0.
\]
By Assumption~\ref{as:ratext}, the set $\mathscr{U}$
is updated with the newly added constraints (for $i \in \mc{N}$)
\[
f_i(p_i(x),x_{-i})-f(x_i, x_{-i}) \,  \ge \,  0 .
\]
The point $u$ does not belong to $\mathscr{U}$ for all future loops.
The cardinality of the set $\mathcal{K}\setminus\mathscr{U}$
decreases at least by one, after each loop.
Note that $\mathscr{U} \subseteq \mc{K}$.
Therefore, if $\mathcal{K}\setminus\mathcal{S}$ is a finite set,
then Algorithm~\ref{ag:KKTSDP} must terminate
within finitely many loops.

Next, suppose Algorithm \ref{ag:KKTSDP}
terminates with a minimizer $u$ in Step~2.
Then $\delta_i\ge 0$ for all $i$,
so every $u_i$ is a minimizer of $\mbox{F}_i(u_{-i})$,
i.e., $u$ is a GNE.
\end{proof}

In Theorem~\ref{cvg:K-S:finite},
the set $\mc{K}\setminus\mc{S}$ being finite is a genericity assumption.
For GNEPs given by generic polynomials, there are finitely many KKT points.
This is shown in the recent work~\cite{Nie2022degree}.
For GNEPs given by generic rational functions.
This can be shown by a similar argument as in \cite[Theorem~3.1]{Nie2022degree}.
Moreover, we remark that the cardinality $|\mc{K}\setminus \mc{S}|$
is only an upper bound for the number of loops taken by Algorithm~\ref{ag:KKTSDP}.
This bound is certainly not sharp, because
the inequality constraint (\ref{ineq:f(p)-f(x)}) may preclude several
(or even all) KKT points that are not GNEs.
In our numerical experiments, Algorithm~\ref{ag:KKTSDP}
often terminates within a few loops.

For some special problems, the KKT point set may be infinite.
When the complement set $\mc{K} \setminus \mc{S}$ is infinite,
Algorithm~\ref{ag:KKTSDP} may not be guaranteed to terminate
within finitely many loops.
However, we can prove its asymptotic convergence under certain assumptions.
For each $i=1, \ldots,N$, we define the $i$th player's value function
\be  \label{eq:valuefunc}
\nu_i(x_{-i}) \, \coloneqq \,
\inf_{x_i\in X_i(x_{-i})} \, f_i(x_i, x_{-i}) .
\ee
The function $\nu_i(x_{-i})$ is continuous
under certain conditions, e.g.,
under the restricted inf-compactness (RIC) condition
(see \cite[Definition 3.13]{GuoLinYeZhang}).
A sequence of functions $\{\phi^{(k)}(x)\}$ is said to be
\textit{uniformly continuous} at a point $x^*$ if for each $\epsilon>0$,
there exists $\tau >0$ such that
$\Vert \phi^{(k)}(x)-\phi^{(k)}(x^*)\Vert<\epsilon$
for all $k$ and for all $x$ with $\| x - x^* \| < \tau$.
The following is the asymptotic convergence result.

\begin{theorem}
\label{thm:alg:convergence}
For the GNEP \reff{eq:GNEP},
suppose Lagrange multipliers can be expressed as in \reff{eq:as:ratexp}
and Assumption \ref{as:ratext} holds for every triple
$(u,i,v_i)$ produced by Algorithm~\ref{ag:KKTSDP}.
In the $k$th loop, let $u^{(k)}$, $v_i^{(k)}$ be the minimizers of
\reff{eq:KKTwithpolyext}, \reff{eq:checkopt:alg}
respectively and let $p_i^{(k)}$ be the feasible extension in Step~3.
Suppose $u^* \coloneqq (u_1^*,\ldots,u_N^*)$ is an accumulation point of
the sequence $\{ u^{(k)} \}_{k=1}^\infty$.
If for each $i=1,\ldots,N$,
\begin{enumerate}

\item[i)] the strict inequality $g_{i,j}(u^*) >0$
holds for all $j \in \mc{I}^{(i)}_2$, and

\item[ii)]
the value function $\nu_i(x_{-i})$ is continuous at $u_{-i}^*$, and

\item[iii)]
the sequence of feasible extensions $\{ p_i^{(k)} \}_{k=1}^\infty$
is uniformly continuous at $u^*$,
\end{enumerate}
then $u^*$ is a GNE for (\ref{eq:GNEP}).
\end{theorem}
\begin{proof}
Up to the selection of a subsequence, we assume that
$u^{(k)}\rightarrow u^*$ as $k\rightarrow \infty$,
without loss of generality.
The condition i) implies that $u^*\in X$ and $u_i^*\in X_i(u_{-i}^*)$
for every $i$. We need to show that each $u_i^*$
is a minimizer for the optimization $\mbox{F}_i(u_{-i}^*)$.
By the definition of $\nu_i$ as in (\ref{eq:valuefunc}),
this is equivalent to showing that
\begin{equation}
\label{eq:nu_diff}
\nu_i(u_{-i}^*)-f_i(u^*)\ge 0,\quad
i = 1, \ldots, N.
\end{equation}
For convenience of notation, let $p_i^{(k)}(x)=x_i$
for each $i\not\in\mathcal{N}$, in the $k$th loop.
Since $u^{(k)}$ is feasible for \reff{eq:KKTwithpolyext}
in all previous loops, we have that
\[
f_i(p_i^{(k')}(u^{(k)}),u_{-i}^{(k)})-f_i(u^{(k)})\ge 0,
\quad \mbox{for all} \quad   k'\le k.
\]
As $k\rightarrow\infty$, the above implies that
\[
f_i(p_i^{(k')}(u^{*}),u_{-i}^*)-f_i(u^*)\ge 0,
\quad \mbox{for all} \quad  k' .
\]
Then, for every $i$ and for every $k\in\mathbb{N}$,
\begin{equation}
\label{eq:nu_diffstar}
\begin{aligned}
&\qquad \nu_i(u_{-i}^*)-f_i(u^*)  \\
&=\big(\nu_i(u_{-i}^*)-f_i(p_i^{(k)}(u^{*}),u_{-i}^*)\big)+
\big(f_i(p_i^{(k)}(u^{*}),u_{-i}^*)-f_i(u^*)\big)  \\
& \ge \nu_i(u_{-i}^*)-f_i(p_i^{(k)}(u^{*}),u_{-i}^*).
\end{aligned}
\end{equation}
Note that $\nu_i(u_{-i}^{(k)})=f_i(p_i^{(k)}(u^{(k)}),u_{-i}^{(k)})$
for all $k$ and for all $i\in\mc{N}$ in the $k$th loop.
Indeed, this is clear by construction when $i\in\mc{N}$.
For $i\notin\mc{N}$, we know $u_i^{(k)}$ is a minimizer for $F_i(u_{-i}^{(k)})$.
Let $p_i^{(k)}(x)=x_i$, then
\[
\nu_i(u_{-i}^{(k)})=f_i(u_i^{(k)},u_{-i}^{(k)})=f_i(p_i^{(k)}(u^{(k)}),u_{-i}^{(k)}).
\]
Under the continuity assumption of $\nu_i$ at $u_{-i}^*$,
the convergence $u^{(k)} \to u^*$ implies that
\[
\nu_i(u_{-i}^*)=\lim_{k\rightarrow\infty}\nu_i(u_{-i}^{(k)})
=\lim_{k\rightarrow\infty} f_i(p_i^{(k)}(u^{(k)}),u_{-i}^{(k)}).
\]
Because $\{ p_i^{(k)} \}_{k=1}^\infty$
is uniformly continuous at $u^*$,
for every fixed $\epsilon>0$, there exists $\tau>0$
such that for all $k$ big enough, we have
\[
\|u^*-u^{(k)}\|< \tau,\quad
\Vert p_i^{(k)}(u^{*})-p_i^{(k)}(u^{(k)})\Vert <\epsilon.
\]
Since $f_i$ is rational and the denominator is positive on $X$,
we have
\[
f_i(p_i^{(k)}(u^{*}),u_{-i}^*)-f_i(p_i^{(k)}(u^{(k)}),u_{-i}^{(k)})
\rightarrow 0\quad \text{as}\quad k\rightarrow\infty.
\]
In view of the inequality (\ref{eq:nu_diffstar}),
we can conclude that $\nu_i(u_{-i}^*)-f_i(u^*)\ge 0$.
This shows that $u^*$ is a GNE.
\end{proof}

When there are strict inequality constraints (i.e., $\mc{I}_2^{(i)}\not=\emptyset$),
the RIC condition is more subtle to check but it is still applicable.
Please note that the strict inequality
$g_{i,j}(x_i, x_{-i}) > 0$ is equivalent to
\[
g_{i,j}(x_i, x_{-i})\cdot  (z_{i,j})^2 = 1,
\]
for a new variable $z_{i,j}$. Similarly, rational functions can be equivalently
reformulated as polynomials by introducing new variables.
Therefore, the value function $\nu_i(x_{-i})$
can be equivalently expressed as the optimal value
of a polynomial optimization problem
with weak inequalities only, in a higher dimensional space.
If the RIC holds for the new formulation,
then one can show the continuity of $\nu_i(x_{-i})$.
There exist some conveniently checkable conditions
for RIC (e.g., see \cite[\S6.5.1]{Clarke}).
For instance, this is the case if the feasible set is compact
or the objective satisfies some growth conditions.
However, checking RIC directly for the rational optimization with strict inequality constraints is typically difficult.
This issue is outside the scope of this paper.

Feasible extensions are sometimes given by polynomials.
For such cases, a sufficient condition for the condition iii)
of Theorem~\ref{thm:alg:convergence} to hold is
that the degrees and coefficients of $\{p_i^{(k)}\}_{k=1}^{\infty}$ are uniformly bounded.
As shown in Subsection~\ref{ssc:commonFE}, when $F_i(x_{-i})$
has box, simplex or ball constraints, feasible extensions have explicit expressions,
and the corresponding polynomial function sequence
$\{p_i^{(k)}\}_{k=1}^{\infty}$ has uniformly bounded degrees and coefficients.
For rational feasible extensions, the condition iii)
is harder to check that needs to be checked case by case.

We would like to remark that
Theorems~\ref{tm:KKTSDPconvex} and \ref{cvg:K-S:finite}
only give sufficient conditions for Algorithm~\ref{ag:KKTSDP}
to terminate within finitely many loops.
But these conditions are not necessary.
In other words, Algorithm~\ref{ag:KKTSDP} may still have finite convergence even if $|\mc{K}\setminus\mc{S}|=\infty$.
This is because the positive definite matrix $\Theta$ is generically selected
(so the optimization \reff{eq:KKTwithpolyext} has a unique minimizer)
and feasible extensions may preclude several
(or even all) KKT points that are not GNEs.
We refer to Example~\ref{ep:quad1}(i)-(ii) for such cases.
When Algorithm~\ref{ag:KKTSDP} does not terminate within finitely many loops,
Theorem~\ref{thm:alg:convergence} proves the asymptotic convergence
under certain assumptions.
We would like to remark that Algorithm~\ref{ag:KKTSDP}
does not need to check if these assumptions are satisfied or not,
because it is self-verifying.
By solving the optimization \reff{eq:checkopt:alg} for each player,
we get a candidate GNE and then verify if it is a true GNE or not.
This does not require checking any other assumptions.

\section{Feasible extensions of KKT points}
\label{sc:RP}

In this section, we discuss the existence and computation
of feasible extensions $p_i$ required as in Assumption~\ref{as:ratext}.
They are important for solving GNEPs.

\subsection{Some common cases}
\label{ssc:commonFE}
The feasible extensions in Assumption~\ref{as:ratext} can be explicitly given for
some common cases of optimization problems.
Suppose the triple $(u,i,v_i)$ is given.

\medskip
\noindent
\textit{Box constraints} \,
Suppose the feasible set of $\mbox{F}_i(x_{-i})$ is
\[
a(\xmi)\le A(\xmi) x_i\le b(\xmi),
\]
where $a,b\in \mathbb{R}[\xmi]^{m_i},\,A\in \mathbb{R}[x_{-i}]^{m_i\times n_i}$.
Suppose $A(x_{-i})$ has full row rank for all $x \in X$
and there is a matrix polynomial $B_0(x_{-i})$ such that
\[
B(x_{-i}) \coloneqq \begin{bmatrix}
A(x_{-i})^T & B_0(x_{-i})
\end{bmatrix}\in \mathbb{R}[x_{-i}]^{n_i\times n_i}
\]
is nonsingular for all $x \in X$.
Let $\mu \coloneqq (\mu_{1},\,\ldots,\,\mu_{m_i})$ be the vector such that
\[
\big(b_j(u_{-i})-a_j(u_{-i})\big)\cdot\mu_j
= b_j(u_{-i})-(B(u_{-i})^Tv_i)_j.
\]
For the case $a_j(u_{-i})=b_j(u_{-i})$,
we just let $\mu_j=0$.
Since $v_i\in X_i(u_{-i})$,
it is clear that each $\mu_j \in [0,1]$.
Then we choose $p_i$ as
\begin{equation}
\label{eq:boxext}
p_i=B(x_{-i})^{-T}\hat{p}_i,
\end{equation}
where $\hat{p}_i=(\hat{p}_{i,1},\ldots, \hat{p}_{i,n_i})$ is defined by
\[
\hat{p}_{i,j}(x) \coloneqq \left\{
\begin{array}{ll}
\mu_ja_j(x_{-i})+(1-\mu_j)b_j(x_{-i}), &  1\le j\le m_i\\
(B(x_{-i})^Tx)_j. & m_i+1\le j\le n_i
\end{array}\right.
\]
One can check that $p_i( u)=v_i$ and
$p_i(x)\in X_i(x_{-i})$ for all $x\in \mathcal{K}\subseteq X$.

We would like to make some remarks for the existence of $B_i(x_{-i})$
that is nonsingular for all $x\in X$.
When $A_i(x_{-i})=A_i$ is independent with $x_{-i}$,
such a constant matrix $B_i$ always exists.
When $A_i(x_{-i})$ is dependent with $x_{-i}$, we may still have such a $B_i(x_{-i})$.

\begin{example}
Consider the two-player's GNEP with $x_1\in\re^1,x_2 = (x_{2,1},x_{2,2})\in\re^2$.
Suppose $X_1(x_2) = \{x_1: (x_1)^2 \le \|x_2\|^2\}$
and $X_2(x_1)$ is given by the inequalities
\[
0\le \underbrace{\bbm x_1 & 1+x_1\ebm}_{A(x_1)} \bbm x_{2,1}\\x_{2,2}\ebm\le 3-x_1.
\]
The $A(x_1)$ has full row rank for all $x\in X$.
We can construct
\[
B(x_1) = \left[\begin{array}{rr}
x_1 & x_1-1\\ 1+x_1 & x_1
\end{array}\right]
\]
such that $\det(B(x_1)) = (x_1)^2-((x_1)^2-1) = 1$.
Therefore, the matrix $B(x_1)$ is nonsingular for all $x_1 \in \re^1$.
\end{example}

\medskip
\noindent
\textit{Simplex constraints}\,
Suppose the feasible set $X_i(x_{-i})$ is given as
\[
d(x_{-i})^Tx_i\le b(x_{-i}),\quad c_j(x_{-i}) x_{i,j} \ge a_j(x_{-i}),\, j\in[n_i].
\]
In the above, $b\in\mathbb{R}[x_{-i}]$, $a = (a_1,\ldots, a_{n_i})$,
$c = (c_1,\ldots,c_{n_i})$ and $d$ are vectors of polynomials in $x_{-i}$.
Assume $c(x_{-i}),d(x_{-i})>0$ for all $x=(x_i,x_{-i})\in X$.
For convenience, use $\odot$ to denote the entrywise product, i.e.,
\[
(c^{-1} \odot a)(x_{-i}) := \left( c_1^{-1}(x_{-i})a_1(x_{-i}),\ldots,
c_{n_i}^{-1}(x_{-i})a_{n_i}(x_{-i})
\right)^T.
\]
Let $\mu \coloneqq (\mu_1,\ldots,\mu_{n_i})$ be vector such that
\[
\big((b-d^Tc^{-1}\odot a)(u_{-i})\big)\cdot\mu_j=v_{i,j}-(c_j^{-1}a_j)(u_{-i}).
\]
For the case that $b(u_{-i})=(d^Tc^{-1}\odot a)(u_{-i})$, just choose $\mu_j=0$.
For $v_i\in X_i(u_{-i})$,  each $\mu_j \in [0,1]$.
Then we choose $p_i \coloneqq (p_{i,1},\ldots,p_{i,n_i})$ such that
\begin{equation}
\label{eq:simpext}
p_{i,j}(x)=\mu_j\cdot\big((b-d^Tc^{-1}\odot a)(x_{-i})\big)+(c_j^{-1}a_j)(x_{-i}).
\end{equation}
One can check that $p_i(u)=v_i$ and
$p_i(x)\in X_i(x_{-i})$ for all $x\in \mathcal{K}\subseteq X$.

\medskip
\noindent
\textit{Ball constraints} \,
Suppose $X_i(\xmi)$ is given as
\[
{\sum}_{j=1}^{n_i}\big( a_j(\xmi)x_{i,j}-c_j(\xmi) \big)^2 \le (R(\xmi))^2,
\]
where $R\in\mathbb{R}[\xmi]$, and
$a = (a_1,\ldots,a_{n_i}),\,c = (c_1,\ldots,c_{n_i})$
are vectors of rational functions in $x_{-i}$.
Assume $a_j(x_{-i})\not=0$ on $X$.
Let $\mu$ be such that
\[
\Vert a(u_{-i}) \odot v_i-c(u_{-i})\Vert = \mu |R(\umi)|,\quad
0 \le \mu \le 1 .
\]
Then choose scalars $(s_1,\ldots,s_{n_i})$ such that
\[
\|a(u_{-i})\odot v_i-c(u_{-i})\|  \cdot s_j =
a_j(u_{-i}) v_{i,j}-c_j(u_{-i}).
\]
For the case $\|a(u_{-i})\odot v_i-c(u_{-i})\| =0$, just let $s_j=1/\sqrt{n_i}$.
Then we can choose $p_i \coloneqq (p_{i,1},\ldots, p_{i,n_i})$ as
\be \label{eq:annular}
p_{i,j}(x) \coloneqq
\big( c_j(x_{-i})+s_j\cdot  \mu \cdot  R(\xmi) \big) / a_j(x_{-i}).
\ee
One can verify that $p_i(u)=v_i$ and $p_i(x)\in X_i(x_{-i})$
for all $x\in \mathcal{K}\subseteq X$.

\subsection{The existence of feasible extensions}
\label{ssc:fe}

The existence of rational feasible extensions
in Assumption~\ref{as:ratext} can be shown under some assumptions.
We consider the general case that the KKT set
$\mc{K}$ as in \reff{eq:KKTweak} is finite.
A polynomial feasible extension $p_i$
exists when $\mathcal{K}$ is finite.

\begin{theorem}     \label{thm:finitK}
Assume $\mathcal{K}$ is a finite set.
Then, for every triple $(u, i, v_i)$ with
$u  \in\mathcal{K}$, $i\in [N]$ and $v_i\in X_i(u_{-i})$,
there must exist a feasible extension $p_i$
satisfying Assumption~\ref{as:ratext}.
Moreover, such $p_i$ can be chosen as a polynomial vector function.
\end{theorem}
\begin{proof}
Since the set $\mathcal{K}$ is finite, by polynomial interpolation,
there must exist a real polynomial vector function $p_i$ such that
\begin{equation}
\label{eq:polyint}
p_{i}(u) = v_{i},\quad
p_{i}(z) = z_i \quad \mbox{for all} \quad
z \coloneqq (z_1, \ldots, z_N)
\in \mc{K}\setminus  \{u\}.
\end{equation}
Note that $\mc{K} \subseteq X$.
For every $x = (x_1, \ldots, x_N) \in \mc{K}\setminus\{u\}$,
we have $p_i(x) = x_i \in X_{i}(x_{-i})$.
The polynomial function $p_i$ satisfies Assumption~\ref{as:ratext}.
\end{proof}

When the set $\mathcal{K}$ is known,
we can get a polynomial feasible extension $p_i$
as in Theorem~\ref{thm:finitK}, by polynomial interpolation.
The following is such an example.

\begin{example}\rm
Consider Example~\ref{eq:simpsimplex}.
There are four KKT points:
\[\begin{array}{ll}
u^{(1)}_1=u^{(1)}_2=(0,0), &
u^{(2)}_1=u^{(2)}_2=\left(\frac{\sqrt{17}-3}{4},\frac{5-\sqrt{17}}{4}\right), \\
u^{(3)}_1=u^{(3)}_2=\left(\frac{1}{2},0\right), &
u^{(4)}_1=u^{(4)}_2=\left(0,\frac{1}{2}\right). \\
\end{array}\]
The $u^{(1)} = (u_1^{(1)}, u_2^{(1)})$ and
$u^{(2)} = (u_1^{(2)}, u_2^{(2)})$ are not GNEs.
For $u^{(1)}$, there are two minimizers for $\mbF_1(u^{(1)}_2)$,
which are $(1,0)$ and $(0,1)$.
We can construct the feasible extension $p_1$ of $(1,0)$ at $u^{(1)}$
using polynomial interpolation.
Consider a linear function $p_1$ such that
\[
p_1 = (
a_0+a_1x_{1,1}+a_2x_{1,2}+a_3x_{2,1}+a_4x_{2,2},\,
b_0+b_1x_{1,1}+b_2x_{1,2}+b_3x_{2,1}+b_4x_{2,2}).
\]
The equation \reff{eq:polyint} requires that
\[
p_1(u_1^{(1)},u_2^{(1)}) = (1,0),\quad
p_1(u_1^{(k)},u_2^{(k)}) = u_1^{(k)},\, k=2,3,4.
\]
This gives a linear system about coefficients of $p_1$:
\[
\begin{array}{l}
a_0 = 1,\, b_0 = 0,\,\\
a_0+\frac{1}{2}a_1+\frac{1}{2}a_3 = \frac{1}{2},\,
b_0+\frac{1}{2}b_1+\frac{1}{2}b_3 = 0,\\
a_0+\frac{1}{2}a_2+\frac{1}{2}b_4 = 0,\,
b_0+\frac{1}{2}b_2+\frac{1}{2}b_4 = \frac{1}{2},\\
a_0+\frac{\sqrt{17}-3}{4}a_1+\frac{5-\sqrt{17}}{4}a_2+
\frac{\sqrt{17}-3}{4}a_3+\frac{5-\sqrt{17}}{4}a_4 = \frac{\sqrt{17}-3}{4},\\
b_0+\frac{\sqrt{17}-3}{4}b_1+\frac{5-\sqrt{17}}{4}b_2+
\frac{\sqrt{17}-3}{4}b_3+\frac{5-\sqrt{17}}{4}b_4 = \frac{5-\sqrt{17}}{4}.
\end{array}
\]
The above linear system is consistent
and we get the feasible extension
\[
p_1(x_1,x_2) = (1-x_{1,1}-x_{1,2}-x_{2,2}, \quad x_{2,2}).
\]
Similarly, we can also get the feasible extension of
$(0,1)$ at $u^{(1)}$,  which is
\[
(x_{1,1}, \quad  1-x_{2,1}-x_{2,2}-x_{1,1} ).
\]
At the point $u^{(2)}$,
the minimizer of $\mbF_1(u^{(2)}_2)$ is $\left(0,\frac{1}{2}\right)$.
We apply polynomial interpolation again.
The linear system in coefficients of $p_1$ is consistent for $\deg(p_1) = 2$.
The following is a feasible extension
\begin{small}
\[
\baray{ll}
\Big(
x_{2,1}(x_{2,1}-\frac{\sqrt{17}-3}{4} )(x_{2,1}+\frac{3+\sqrt{17}}{2(5-\sqrt{17})}), &
\frac{1}{2}-(x_{2,2}-\frac{1}{2})(x_{2,2}-\frac{5-\sqrt{17}}{4})
(x_{2,2}+\frac{4}{5-\sqrt{17}})
\Big).
\earay
\]
\end{small}
\end{example}

When the set $\mathcal{K}$ is not finite,
Assumption~\ref{as:ratext} may still hold for some GNEPs.
For instance, consider that there are no equality constraints, i.e.,
$\mc{I}^{(i)}_0 =\emptyset$. Suppose $\mathcal{K}$ is compact
and there exists a continuous map
$\rho: \mathbb{R}^n \to \mathbb{R}^{n_i}$ such that $\rho(u)=v_i$
and $g_{i,j}(\rho(x),x_{-i})>0$ for all $x\in\mathcal{K}$
and for all $j\in \mc{I}^{(i)}_1 \cup \mc{I}^{(i)}_2$.
For every $\epsilon>0$, one can approximate $\rho$
by a polynomial $p_i$ such that
$\|p_i-\rho \| < \epsilon$ on $\mathcal{K}$.
Therefore, for $\epsilon$ sufficiently small,
$g_{i,j}(p_i(x),x_{-i}) > 0$ on $x\in\mathcal{K}$.
Such a polynomial function $p_i$ is a feasible extension of $v_i$ at $u$.

\subsection{Computation of feasible extensions}
\label{ssc:cfe}

We discuss how to compute the rational feasible extension $p_i$ satisfying
Assumption~\ref{as:ratext}.
For the set $\mc{K}$ as in \reff{eq:KKTweak},
let $E_0$ denote the set of its equality constraining polynomials
and let $E_1$ denote the set of its
(both weak and strict) inequality ones.
Consider the set
\[
\mathcal{K}_1 \, \coloneqq \,
\left \{ x\in\re^n
\left|\baray{l}
 g(x)=0\,(g\in E_0), \\
g(x)\ge 0\,( g\in E_1 )
\earay\right.
\right \} .
\]
The set $\mc{K}$ may not be closed but $\mc{K}_1$ is,
and the closure of $\mc{K}$ is contained in $\mc{K}_1$. For a polynomial $p(x)$,
if $p(x) \in X_i(x_{-i})$ for all $x \in \mc{K}_1$,
then we also have $p(x) \in X_i(x_{-i})$ for all $\mc{K}$.
Therefore, it is sufficient to get $p_i$ satisfying
Assumption~\ref{as:ratext} with $\mc{K}$ replaced by $\mc{K}_1$.

Suppose the triple $(u, i, v_i)$ is given.
First, choose a {\it priori} degree $l$, and choose a denominator $h$
that is positive on $\mc{K}$ (e.g., one may choose $h=1$).
Then, we consider the following feasibility problem in $(q, \mu)$
\begin{equation}
\label{eq:feasibleRP}
\left\{
\begin{array}{l}
 q  \coloneqq (q_{1},\ldots, q_{n_i}) \in \mathbb{R}[x]_{2l}^{n_i}, \, \,
 \mu \coloneqq  (\mu_j)_{ j \in  \mc{I}^{(i)}_1 \cup  \mc{I}^{(i)}_2 }, \\
 q(u) = h(u)v_{i},  \,
   h \cdot  g_{i,j}(q,x_{-i}) = 0  \, (j\in \mc{I}^{(i)}_0 ),\\
 \mu_{j}  \ge 0 \,  ( j \in  \mc{I}^{(i)}_1  ), \,
          \mu_{j}  > 0 \,  ( j \in  \mc{I}^{(i)}_2  )          \\
 h \cdot  g_{i,j}( q,x_{-i})-\mu_{j} \in \idl[E_0]_{2l}+\qmod[E_1]_{2l} .
\end{array}
\right.
\end{equation}
When all constraining polynomials $g_{i,j}$ are linear in $x_i$,
the system \reff{eq:feasibleRP} is convex in $(q, \mu)$,
and it ensures that $p_i \coloneqq q/h$
is a rational feasible extension
satisfying Assumption~\ref{as:ratext}.
For such a case, a feasible pair $(q, \mu)$ for (\ref{eq:feasibleRP})
can be obtained by solving a linear conic optimization problem.

\begin{example} \rm
\label{ep:SDPRP:yalmip}
Consider the following $2$-player GNEP:
\be\label{eq:constructRP}
\begin{array}{cllcl}
\min\limits_{x_{1} \in \re^2 }& \frac{(x_{2,1}+x_{2,2}-2x_{1,1})(x_{1,1})^2+2x_{1,2}}{x_{2,1}}&\vline&
\min\limits_{x_{2} \in \re^2 }& \frac{x_{2,1}-(x_{2,2})^2}{x_{2,2}+x_{1,1}+x_{1,2}} \\
\st & 2x_{1,1}x_{2,1}-x_{1,2}x_{2,2}\ge0, &\vline& \st &2x_{2,1}x_{2,2}-1\ge 0,\\
&x_{2,1}x_{2,2}-x_{1,1}x_{2,1}\ge 0, &\vline& & 1-x_{2,2}  \ge 0, \\
& 2x_{1,2}x_{2,2}-1\ge 0, & \vline & & 2-x_{2,1} \ge 0, \\
& 2-x_{1,2}x_{2,2}\ge 0 ; & \vline & & x_{2,1}  \ge 0 .
\end{array}
\ee
Consider the triple $(u,1,v_1)$ for $u=(u_1,u_2)$ with
\[
u_1=(0.5,0.5),\quad u_2=(0.5,1),\quad v_1 = (1,0.5).
\]
For $l=2$ and $h = x_{2,1}x_{2,2}$,
a feasible $q$ given by (\ref{eq:feasibleRP}) is $(x_{2,2}, x_{2,1})/2$.
Let $p_1 = \frac{1}{ 2x_{2,1}x_{2,2}} (x_{2,2}, x_{2,1}).$
Then we have each
$h \cdot g_{1,j}(p_1,x_2) \in \idl[E_0]_{2l}+ \qmod[E_1]_{2l}$:
\[
\begin{array}{l}
h \cdot g_{1,1}(p_1, x_2)=0.25+0.25(2x_{2,1}x_{2,2}-1), \\
h \cdot g_{1,2}(p_1, x_2) =(x_{2,1}x_{2,2}-0.5)^2+0.25(2x_{2,1}x_{2,2}-1),\\
h \cdot g_{1,3}(p_1, x_2)=0,\quad
h \cdot g_{1,4}(p_1, x_2)=0.75+0.75(2x_{2,1}x_{2,2}-1).
\end{array}
\]
\end{example}

For the triple $(u,i,v_i)$, when some constraining polynomials
$g_{i,j}$ are nonlinear in $x_i$,
the system \reff{eq:feasibleRP} may not be convex in $(q,\mu)$.
For such cases, it is not clear how to obtain
feasible extensions in a computationally efficient way.
The existence of such $p_i$
is guaranteed when $\mc{K}$ is a finite set.
This is shown in Theorem~\ref{thm:finitK}.
When $\mc{K}$ is fully known,
we can get the $p_i$ by polynomial interpolation.
For other cases, it is not clear for us
how to compute such $p_i$ efficiently.

\section{Rational optimization problems}
\label{sc:ROP}

This section discusses how to solve the rational optimization problems
appearing in Algorithm~\ref{ag:KKTSDP}.

\subsection{Rational polynomial optimization}
\label{sc:intoRO}

A general rational polynomial optimization problem is
\be \label{eq:RatOpt}
\left\{
\baray{rl}
\min  & A(x) \coloneqq \frac{a_1(x)}{a_2(x)}\\
\st   &  x\in K,
\earay
\right.
\ee
where $a_1,a_2\in\mathbb{R}[x]$ and $K\subseteq\mathbb{R}^n$
is a semialgebraic set. We assume the denominator $a_2(x) > 0$ on $K$,
otherwise one can minimize $A(x)$ over two subsets
$K\cap\{a_2(x) > 0\}$ and $K\cap\{-a_2(x) > 0\}$ separately.
Moment-SOS relaxations can be applied to solve \reff{eq:RatOpt}.
We refer to \cite{GloPol3,JdeK06,NDG08} for related work.
Please note that Lagrange multipliers are zeros
for strict inequality constraints. So the KKT system does not need to
consider strict inequality constraints. However, the strict inequalities
are still used in the Moment-SOS relaxations,
because they are relaxed to weak inequality constraints.

The rational optimization problems in Algorithm~\ref{ag:KKTSDP}
may have strict inequalities.
So we consider the case that $K$ is given as
\begin{equation}
\label{rat:feasC}
K=\left\{x\in\mathbb{R}^n\left|
\begin{array}{c}
p(x) = 0 \,(p\in \Psi_0),\\
q(x) \ge 0\, (q\in \Psi_1),\\
q(x) >0\, (q\in \Psi_2)
\end{array}
\right.
\right\},
\end{equation}
where $\Psi_0,\,\Psi_1$ and $\Psi_2$
are finite sets of constraining polynomials in $x$.
Since $a_2(x) >0$ on $K$, we have $A(x)\ge \gamma$ on $K$
if and only if $a_1(x)-\gamma a_2(x)\ge 0$ on $K$, or equivalently,
$a_1-\gamma a_2\in\mathscr{P}_d(K)$, for the degree
\[
d \coloneqq \max\{\deg(a_1),\deg(a_2)\}.
\]
The rational optimization problem (\ref{eq:RatOpt}) is then equivalent to
\be  \label{eq:P(C)}
\left\{
\begin{aligned}
    \gamma^*: = \max\quad &\gamma\\
    \st \quad & a_1(x)-\gamma a_2(x) \in \mathscr{P}_d(K).
    \end{aligned}
    \right.
\ee
Denote the weak inequality set
\be \label{def:K_1}
K_1 := \left\{
x\in\mathbb{R}^n\left|
\begin{array}{c}
p(x)=0\,(p\in\Psi_0),\\
q(x)\ge 0\,(q\in\Psi_1 \cup \Psi_2)
\end{array}
\right.
\right\}.
\ee
Note that $K_1$ is closed and $\mbox{cl}(K) \subseteq K_1$.
We consider the moment optimization problem
\begin{equation}
\label{eq:R(C)}
    \left\{
    \begin{aligned}
      \min &\quad \langle a_1,w\rangle\\
    \st& \quad \langle a_2,w\rangle = 1,\, w\in\mathscr{R}_d(K_1).
    \end{aligned}
    \right.
\end{equation}
It is a moment reformulation for the optimization
\be \label{ratopt:K1}
 \left\{
 \baray{rl}
a^* \coloneqq \min  & A(x) \\
    \st  &  x\in K_1 .
  \earay
 \right.
\ee
Note that \reff{ratopt:K1} is a relaxation of \reff{eq:RatOpt}.
It is worthy to observe that if a minimizer of \reff{ratopt:K1}
lies in the set $K$, then it is also a minimizer of \reff{eq:RatOpt}.

We apply Moment-SOS relaxations to solve (\ref{eq:R(C)}). Let
\be \label{def:d_0}
d_0 \coloneqq \max\big\{\lceil d/2\rceil,\, \lceil \deg(g)/2\rceil \,
(g \in  \Psi_0 \cup \Psi_1 \cup \Psi_2)  \big\}.
\ee
For an integer $k\ge d_0$, the $k$th order SOS relaxation for (\ref{eq:P(C)}) is
\begin{equation}
\label{eq:SOSrat}
\left\{
\begin{aligned}
\gamma^{(k)} \coloneqq \max\quad & \gamma\\
\st\quad & a_1(x)-\gamma a_2(x) \in
    \idl[\Psi_0]_{2k}+\qmod[\Psi_1 \cup \Psi_2]_{2k}.
\end{aligned}
\right.
\end{equation}
The dual optimization of (\ref{eq:SOSrat})
is the $k$th order moment relaxation
\be \label{eq:Momrat}
\left\{ \baray{rl}
 a^{(k)} \coloneqq \min\quad & \langle a_1,y\rangle\\
\st & L_p^{(k)}[y]=0\,(p\in\Psi_0),\\
 & L_q^{(k)}[y]\succeq 0\,(q\in\Psi_1 \cup \Psi_2), \\
 & \langle a_2,y\rangle = 1,\, M_k[y]\succeq 0,\,
 y\in\mathbb{R}^{\mathbb{N}_{2k}^n} .
\earay \right.
\ee
Since (\ref{eq:Momrat}) is a relaxation of \reff{eq:R(C)},
if (\ref{eq:Momrat}) is infeasible,
then (\ref{eq:RatOpt}) is also infeasible.

The following is the Moment-SOS algorithm for solving (\ref{eq:RatOpt}).
It can be conveniently implemented with
the software {\tt GloptiPoly 3} \cite{GloPol3}.

\begin{alg} \label{ag:ratopt} \rm
For the rational optimization problem (\ref{eq:RatOpt}),
let $k \coloneqq d_0$.

\begin{itemize}

\item [Step~1]
Solve the $k$th order moment relaxation~(\ref{eq:Momrat}).
If it is infeasible, then (\ref{eq:RatOpt}) has no feasible points and stop.
Otherwise, solve it for the optimal value $ a^{(k)}$ and a minimizer $y^*$,
if they exist. Let $t \coloneqq d_0$ and go to Step~2.

\item [Step~2]
Check whether or not there is an order $t \in [d_0, k]$ such that
\be \label{eq:flatranklow}
r \coloneqq \Rank{M_t[y^*]} \,=\, \Rank{M_{t-d_0}[y^*]} .
\ee

\item [Step~3] If (\ref{eq:flatranklow}) fails,
let $k \coloneqq k+1$ and go to Step~1;
if \reff{eq:flatranklow} holds, find points $z_1,\ldots,z_r\in K_1$
and scalars $\mu_1,\ldots,\mu_r>0$ such that
\begin{equation}
\label{eq:ydecomp}
y^*|_{2t} = \mu_1[z_1]_{2t}+\cdots+\mu_r[z_r]_{2t} .
\end{equation}

\item [Step~4]
Output each $z_i \in K$ with $a_2(z_i)>0$
as a minimizer of (\ref{eq:RatOpt}).

\end{itemize}
\end{alg}

In Step~2, the rank condition (\ref{eq:flatranklow})
is called \textit{flat truncation}.
It is sufficient and almost necessary for checking convergence
of the Moment-SOS hierarchy  (see \cite{nie2013certifying}).
Once (\ref{eq:flatranklow}) is met,
the moment relaxation (\ref{eq:Momrat}) is tight for solving (\ref{eq:R(C)}),
and the decomposition~\reff{eq:ydecomp} can be
computed by the Schur decomposition \cite{HenLas05}.
This is also implemented in the software {\tt GloptiPoly 3} \cite{GloPol3}.
When $\idl[\Psi_0]+\qmod[\Psi_1 \cup \Psi_2]$ is archimedean,
one can show that $ a^{(k)}\rightarrow a^*$
as $k\rightarrow\infty$ (see \cite{nie2015linear}).
The following is the justification for the conclusion in Step~4.

\begin{theorem}
\label{thm:RO:conv}
Assume $a_2 \ge 0$ on $K_1$.
Suppose $y^*$ is a minimizer of (\ref{eq:Momrat})
and it satisfies (\ref{eq:flatranklow}) for some order $t\in[ d_0, k]$.
Then, each $z_i$ in (\ref{eq:ydecomp}),
such that $a_2(z_i)>0$ and $z_i \in K$,
is a minimizer of (\ref{eq:RatOpt}).
\end{theorem}
\begin{proof}
Under the rank condition (\ref{eq:flatranklow}),
the decomposition (\ref{eq:ydecomp}) holds for some points
$z_1, \ldots, z_r \in K_1$ (see \cite{HenLas05,nie2013certifying}).
The constraint $\langle a_2, y^* \rangle = 1$ implies that
\[
1 = \langle a_2, y^* \rangle = \mu_1 a_2(z_1) + \cdots + \mu_r a_2(z_r) .
\]
Since $a_2 \ge 0$ on $K_1$, we know all $a_2(z_j) \ge  0$.
Let $J_1 \coloneqq \{ j: a_2(z_j) > 0  \}$ and
$J_2 \coloneqq \{ j: a_2(z_j) = 0  \}$, then
\[
\langle a_1, y^* \rangle =
\sum_{ j \in J_1  } \mu_j a_2(z_j)  A(z_j) +
\sum_{ j \in J_2  } \mu_j a_1(z_j)  .
\]
Note that $\sum_{ j \in J_1  } \mu_j a_2(z_j)  = 1$
and each $[z_j]_{2k} \in \mathscr{R}_{2k}(K_1)$.
For all nonnegative scalars $\nu_j \ge 0$, $j \in J_1 \cup J_2$
such that $\sum_{ j\in J_1} \nu_j a_2(z_j) = 1$, the tms
\[
z(\nu) \coloneqq  \nu_1[z_1]_{2k}+\cdots+\nu_r[z_r]_{2k}
\]
is a feasible point for the moment relaxation \reff{eq:Momrat}.
Therefore, the optimality of $y^*$ implies that
$A(z_j) =  a^{(k)}$ for all $j \in J_1.$
Since $a^{(k)} \le a^*$ and each $z_j \in K_1$,
we have $A(z_j) \ge a^*$. Hence,
$A(z_j) = a^*$ for all $j \in J_1.$
Note that \reff{eq:R(C)} is a relaxation of \reff{ratopt:K1}.
So each $z_j$ ($j \in J_1$) is a minimizer of \reff{ratopt:K1}.
Therefore, every $z_i \in K$ satisfying $a_2(z_i)>0$
is a minimizer of (\ref{eq:RatOpt}).
\end{proof}

In the decomposition (\ref{eq:ydecomp}), it is possible that
no $z_i$ belongs to the set $K$.
This is because the feasible set $K$ may not be closed,
due to strict inequality constraints.
For such a case, the optimal value of (\ref{eq:RatOpt})
may not be achievable.
If we obtain a minimizer $y^*$ of \reff{eq:Momrat}
such that $\rank M_k[y^*]$ is maximum
and (\ref{eq:flatranklow}) is satisfied,
then we can get all minimizers of \reff{ratopt:K1}.
Moreover, if \reff{ratopt:K1} has infinitely many minimizers,
the rank condition (\ref{eq:flatranklow}) cannot be satisfied easily.
We refer to \cite{Lau09,nie2013certifying} for this fact.
When primal-dual interior point methods are used to solve \reff{eq:Momrat},
a minimizer $y^*$ with $\rank \, M_k[y^*]$
maximum is often returned.
Therefore, if \reff{ratopt:K1} has finitely many minimizers
and primal-dual interior point methods are used,
then some points $z_i$ \reff{eq:ydecomp} must belong to the set $K$.
This means that we can typically find all minimizers of
\reff{eq:RatOpt} and \reff{ratopt:K1},
even if there are strict inequality constraints.
However, if the optimal value of (\ref{eq:RatOpt}) is not achievable,
then no $z_i$ in \reff{eq:ydecomp} belongs to $K$.
We refer to \cite{GloPol3,JdeK06,NDG08} for the work on
solving rational optimization problems.

\subsection{The optimization for all players}
\label{sc:opt4all}

The rational optimization problem
in Step~2 of Algorithm \ref{ag:KKTSDP} is
\begin{equation}
\label{ratopt:forall}
\left\{
\baray{rl}
\min & \theta(x) \coloneqq [x]_1^T\Theta[x]_1\\
 \st\quad & x\in\mathscr{U},
\earay
\right.
\end{equation}
where $\Theta$ is a generic positive definite matrix.
The feasible set $\mathscr{U}$ can be expressed as in the form \reff{rat:feasC},
with polynomial equalities and weak/strict inequalities,
for some polynomial sets $\Psi_0,\Psi_1,\Psi_2$.
That is, \reff{ratopt:forall} can be expressed in the form of \reff{eq:RatOpt},
with {denominators} being $1$.
Denote the corresponding set
\begin{equation}
\label{def:mathscrK_1}
\mathscr{U}_1 = \{x\in\mathbb{R}^n|\, p(x)=0\,(p\in\Psi_0),\,
q(x)\ge 0\,(q\in\Psi_1 \cup \Psi_2)\}.
\end{equation}
Since $\Theta$ is positive definite,
the objective $\theta$ is coercive and strictly convex.
When $\Theta$ is also generic,
the function $\theta$ has a unique minimizer $u^*$ on the set $\mathscr{U}_1$
if it is nonempty. Suppose $y^*$ is a minimizer of
the $k$th order moment relaxation of (\ref{ratopt:forall}).
Then, in Algorithm \ref{ag:ratopt},
the rank condition (\ref{eq:flatranklow}) is reduced to
\[
\Rank{M_t[y^*]} = 1
\]
for some order $t \in [d_0, k]$
and the decomposition (\ref{eq:ydecomp})
is equivalent to $y^*|_{2t} = \mu_1 [z_1]_{2t}$
for some $z_1 \in \mathscr{U}_1$.
Algorithm~\ref{ag:ratopt} can be applied to solve (\ref{ratopt:forall}).
The following are some special properties of
Moment-SOS relaxations for (\ref{ratopt:forall}).

\begin{theorem}
\label{thm:ratforall}
Assume $\Theta$ is a generic positive definite matrix.
\begin{enumerate}

\item[i)]
If the set $\mathscr{U}_1$ is empty and
$\idl[\Psi_0]+\qmod[\Psi_1 \cup \Psi_2]$ is archimedean,
then the moment relaxation for (\ref{ratopt:forall})
must be infeasible when the order $k$ is big enough.

\item[ii)]
Suppose $\mathscr{U}_1 \ne \emptyset$ and
$\idl[\Psi_0]+\qmod[\Psi_1 \cup \Psi_2]$ is archimedean.
Let $u^{(k)} \coloneqq  (y_{e_1}^{(k)},\ldots, y_{e_n}^{(k)})$,
where $y^{(k)}$ is the minimizer of the $k$th order moment relaxation
 of (\ref{ratopt:forall}). Then $u^{(k)}$ converges to the unique minimizer of
$\theta$ on $\mathscr{U}_1$.
	
\item[iii)]
Suppose the real zero set of $\Psi_0$ is finite.
If $\mathscr{U}_1 \ne \emptyset$, then we must have
$\Rank{M_t[y^*]} = 1$ for some $t \in [d_0, k]$,
when $k$ is sufficiently large.
\end{enumerate}
\end{theorem}
\begin{proof}
i) When $\mathscr{U}_1 = \emptyset$,
the constant $-1$ can be viewed as a positive polynomial on $\mathscr{U}_1$.
Since $\idl[\Psi_0]+\qmod[\Psi_1 \cup \Psi_2]$ is archimedean,
we have $-1\in \idl[\Psi_0]_{2k}+\qmod[\Psi_1 \cup \Psi_2]_{2k}$ for $k$ big enough,
by Putinar's Positivstellensatz. For such $k$,
the corresponding SOS relaxation~\reff{eq:SOSrat} is unbounded from above,
and hence the corresponding moment relaxation must be infeasible.

ii) When $\mathscr{U}_1 \ne \emptyset$,
the objective $\theta$ has a unique minimizer $u^*$ on $\mathscr{U}_1$.
The convergence of $u^{(k)}$ is implied by
\cite[Theorem~3.3]{nie2013certifying} (also see \cite{Swg05}).

iii) When the real zero set of $\Psi_0$ is finite
and $\mathscr{U}_1 \ne \emptyset$,
the conclusion can be implied by
\cite[Proposition~4.6]{lasserre2008semidef}
(also see \cite{Lau09}).
\end{proof}

\subsection{Checking Generalized Nash Equilibria}
\label{ssc:GNE}

Once we get a minimizer $u$ of (\ref{ratopt:forall}),
we need to check if it is a GNE or not. For each $i=1,\ldots,N$,
we need to solve the rational optimization problem
\begin{equation}
\label{rat:checkopt}
\left\{ \baray{rl}
\delta_i \coloneqq \min & f_i(x_i,u_{-i})-f_i(u_i, u_{-i}) \\
\st  & x_i \in X_i(u_{-i}),
\earay \right.
\end{equation}
where $f_i,X_i(u_{-i})$ are given in (\ref{eq:GNEP}).
Assume the KKT conditions hold and the Lagrange multiplies
can be expressed as in \reff{eq:as:ratexp},
then (\ref{rat:checkopt}) is equivalent to
\begin{equation}
\label{rat:checkopt:tight}
\left\{
\baray{rl}
\min   & f_i(x_i,u_{-i})-f_i(u_i,u_{-i})\\
\st & \nabla_{x_i}f_i(x_i,u_{-i})  \,=
\sum\limits_{j \in  \mc{I}^{(i)}_0 \cup  \mc{I}^{(i)}_1  }
    \tau_{i,j}(x_i,u_{-i})\nabla_{x_i} g_{i,j}(x_i,u_{-i}),  \\
& \tau_{i,j}(x_i,u_{-i})  g_{i, j}(x_i,u_{-i}) = 0, \,
\tau_{i,j}(x_i,u_{-i})\ge 0 \,(j\in\mc{I}^{(i)}_1), \\
&  x_i \in X_i(u_{-i}) .
\earay
\right.
\end{equation}
We can equivalently express the feasible set of
\reff{rat:checkopt:tight} in the form
\be \label{def:G_i}
Y_i(u_{-i})=\left\{x_i\in\mathbb{R}^{n_i}\left|
\begin{array}{l}
 p(x_i) = 0\,(p\in \Psi_{i,0}),\\
 q(x_i) \ge 0\,(q\in \Psi_{i,1}),\\
 q(x_i) >0\, (q\in \Psi_{i,2})
\end{array}
\right.
\right\},
\ee
for three sets $\Psi_{i,0}, \Psi_{i,1}, \Psi_{i,2}$ of polynomials in $x_i$.
In computational practice, we need to assume \reff{rat:checkopt:tight}
is solvable, i.e., the solution set of \reff{rat:checkopt:tight} is nonempty.
As in the Subsection~\ref{sc:intoRO},
we can apply a similar version of Algorithm \ref{ag:ratopt}
to solve the rational optimization problem (\ref{rat:checkopt:tight}).
Similar conclusions like in Theorem~\ref{thm:ratforall}
hold for the corresponding Moment-SOS relaxations.
A difference is that all rational functions for \reff{rat:checkopt}
are only in the variable $x_i$ instead of $x$.
It may have several different minimizers,
so the rank in \reff{eq:flatranklow} may be bigger than one.
Generally, the optimization \reff{rat:checkopt:tight}
is easier to solve than \reff{ratopt:forall}.

\section{Numerical experiments}
\label{sc:ne}

This section gives numerical experiments for Algorithm~\ref{ag:KKTSDP}
to solve rGNEPs. The rational optimization problems
are solved by Moment-SOS relaxations,
which are implemented with the software {\tt GloptiPoly~3} \cite{GloPol3}.
The semidefinite programs for the Moment-SOS relaxations
are solved by {\tt SeDuMi} \cite{sturm1999using}.
The computation is implemented in MATLAB R2018a,
in a Laptop with CPU 8th Generation Intel® Core™ i5-8250U and RAM 16 GB.
For neatness of the paper, only four decimal digits are
displayed for computational results.
The accuracy for a point $u$ to be a GNE is measured by the quantity
\begin{equation}
\label{eq:dtacc}
\dt \, \coloneqq \, \min \{\dt_1, \ldots, \dt_N \},
\end{equation}
where $\dt_i$ is the optimal value of \reff{eq:checkopt:alg}.
The point $u$ is a GNE if and only if $\dt = 0$.
Due to numerical issues, $u$ can be viewed as a GNE
if $\dt$ is nearly zero (e.g., $\delta \ge -10^{-6}$).
For cleanness of presentation,
we do not list the constraining functions $g_{i,j}$ explicitly.
Instead, they are ordered row by row, from top to bottom;
in each row, they are ordered from left to right.
If there is an inequality like $a(x) \le b(x)$,
then the corresponding constraining function is $b(x)-a(x)$.

To implement Algorithm~\ref{ag:KKTSDP}, we need rational LMEs.
This is reviewed in Subsection~\ref{sc:LME}.
More details can be found in \cite{Nie2019}.
For some standard constraints (e.g., box, simplex or balls),
we can have LMEs explicitly given as follows.

\begin{itemize}
\item[i)] Consider the box constraints
$
a(x_{-i})\le x_i\le b(x_{-i}),
$
where $a=(a_1,\ldots,a_{n_i})$, $b=(b_1,\ldots,b_{n_i})$.
The LME is, for $j=1,\ldots, n_i$,
\be\begin{array}{c}\label{eq:boxLME}
\lambda_{i,2j-1}=\frac{b_j(\xmi)-x_{i,j}}{b_j(\xmi)-a_j(\xmi)}\cdot\frac{\partial f_i}{\partial x_{i,j}},\quad
\lambda_{i,2j}=\frac{x_{i,j}-a_j(\xmi)}{b_j(\xmi)-a_j(\xmi)}\cdot\frac{\partial f_i}{\partial x_{i,j}}.
\end{array}
\ee

\item[ii)] Consider the simplex constraints
$u(\xmi) \ge e^Tx_i$, $x_i \ge l(x_{-i})$,
where $l$ is a vector function in $x_{-i}$.
The LME is $\lambda_i=(\lambda_{i,1},\,\hat{\lambda}_i)$ with
\be\label{eq:simplexLME}
\begin{array}{l}
\lambda_{i,1}=-\frac{(x_i- l(x_{-i}) )^T\nabla_{x_i} f_i}{u(\xmi)-e^Tl(x_{-i})},\quad
\hat{\lambda}_{i}=\nabla_{x_i}f_i+\lambda_{i,1}\cdot e.
\end{array}
\ee

\item[iii)] Consider the ball type constraint
$r(x_{-i})\le \Vert x_i-c\Vert^2\le R(x_{-i}),$
where $c=(c_1,\dots,c_{n_i})$ is a constant vector. The LME is
\be\label{eq:annularLME}
\lambda_{i}=\mbox{\small
$\left(
\frac{R(\xmi)\frac{\partial f_i}{\partial x_{i,1}}-(x_{i,1}-c_1)\cdot(x_i-c)^T\nabla_{\xpi}f_i}{2(x_{i,1}-c_1)(R(\xmi)-r(\xmi))},\,
\frac{r(\xmi)\frac{\partial f_i}{\partial x_{i,1}}-(x_{i,1}-c_1)\cdot(x_i-c)^T\nabla_{\xpi}f_i}{2(x_{i,1}-c_1)(R(\xmi)-r(\xmi))}
\right).$
}
\ee
For the special case that $r(\xmi)=0$,
the LME is reduced to
\be\label{eq:ballLME}
\lambda_{i}=(c-x_i)^T\nabla_{\xpi}f_i/\big(2R(\xmi)\big).
\ee
\end{itemize}

\subsection{Some fractional quadratic GNEPs}
First, we consider rGNEPs with fractional quadratic objectives
and standard constraints (e.g., box, simplex or balls).
These GNEPs often appear in various applications.
We give details for applying Algorithm~\ref{ag:KKTSDP} in such problems.

\begin{example}\label{ep:quad1}\rm
(i) Consider the 2-player rGNEP
\be\label{eq:quad1}
\begin{array}{cllcl}
\min\limits_{x_{1} \in \re^2 }& \frac{-(x_{1,1})^2-x_{2,1}x_{1,1}}{x_{1,2}x_{2,2}+1} &\vline&
\min\limits_{x_{2} \in \re^2 }& \frac{3x_{2,1}x_{2,2}-2 x_{2,2}}{x_{1,2}x_{2,2}+1} \\
\st & (x_{2,1})^2-x_1^Tx_1\ge 0,&\vline&
\st & 0.5\le x_{2,1}\le 1,\\
    & &\vline && 0\le x_{2,2}\le x_{1,1}.\\
\end{array}
\ee
The LME for the first player is in form of \reff{eq:ballLME},
and the LMEs for the second player are given as in \reff{eq:boxLME}.
Precisely, we have
$\lambda_1 = \frac{-x_1^T\nabla_{x_1}f_1}{2(x_{2,1})^2}$ and
\[
\lambda_2 \quad = \quad \Big( (2-2x_{2,1})\frac{\partial f_2}{\partial x_{2,1}}, \quad
(2x_{2,1}-1)\frac{\partial f_2}{\partial x_{2,1}}, \quad
\frac{x_{1,1}-x_{2,2}}{x_{1,1}}\frac{\partial f_2}{\partial x_{2,2}}, \quad
\frac{x_{2,2}}{x_{1,1}}\frac{\partial f_2}{\partial x_{2,2}} \Big).
\]
By applying Algorithm \ref{ag:KKTSDP}, we get
\[
\begin{array}{l|l}
\hline
k = 0 & u^{(0)}_1=u^{(0)}_2=(0.6667, 0000),\\
& \delta_1 = -3.6732\cdot 10^{-7},\quad
  \delta_2 = -0.3333 ,\\
&  v_2^{(0)} = (0.5000,0.6667),\quad p_2^{(0)}(x) = (0.5, x_{1,1}). \\ \hline
k = 1 & u^{(1)}_1=(0.4930, -0.0835),\, u^{(1)}_2=(0.5000, 0.4930),\\
& \delta_1 = -4.3101\cdot 10^{-7},\quad \delta_2 = -8.9324\cdot10^{-9} \\
& \mbox{A GNE is returned, in $4.22$ seconds.} \\ \hline
\end{array}
\]
In the above, $u^{(k)} = (u_1^{(k)}, u_2^{(k)})$, $v_i^{(k)}$
denote the minimizers of
\reff{eq:KKTwithpolyext}-\reff{eq:checkopt:alg} in the $k$th loop.
The $p_2^{(0)}(x)$ is the feasible extension of $v_2^{(0)}$ at $u^{(0)}$,
which is given as in \reff{eq:boxext}.
Interestingly, \reff{eq:quad1} has infinitely many non-GNE KKT points.
Because one can check that $(t,0,t,0)\in \mc{K}\setminus\mc{S}$
for every $t\in[\frac{2}{3},1]$.
However, Algorithm~\ref{ag:KKTSDP} still has finite convergence,
as verified in computational practice.
It implies that the upper bound $|\mc{K}\setminus \mc{S}|$
given in Theorem~\ref{thm:finitK} is not sharp.
In addition, we would like to remark that finite convergence
is guaranteed by the use of feasible extension $p_2(x) = (0.5, x_{1,1})$. Since
\[
f_2(x_1, p_2(x)) - f_2(t,0,t,0) = -0.5t<0,\quad \forall t\in [2/3, \, 1],
\]
then the whole set $\{(t,0,t,0):t\in [\frac{2}{3},1] \}$
can be precluded by (\ref{ineq:f(p)-f(x)}).

\noindent (ii) For the GNEP (\ref{eq:quad1}),
if the first player's objective function is changed to
\[\frac{-(x_{1,1})^2+x_{2,1}x_{1,1}}{x_{1,2}x_{2,2}+1},\]
then Algorithm~\ref{ag:KKTSDP} produces the following computational results:
\[
\begin{array}{l|l}
\hline
k = 0 & u^{(0)}_1=(0.3333,\, -0.3049),\,
   u^{(0)}_2=(0.6667,\, 0.0000),\\
& \delta_1 = -1.0000 ,\quad \delta_2 = -0.1856  ,\\
&  v_1^{(0)} = (-0.6667, \, 0.0000), \quad  p_1^{(0)}(x) = (-x_{2,1}, \, 0), \\
& v_2^{(0)} = (0.5000,\, 0.3333), \quad  p_2^{(0)}(x) = (0.5, \, x_{1,1}).  \\ \hline
k = 1 & \mbox{Nonexistence of GNEs is detected, in 5.56 seconds.}  \\ \hline
\end{array}
\]
Similar to (i), there are infinitely many non-GNE KKT points, which are
$\left(\alpha,\beta,2\alpha,0\right)$ with
\[
\alpha \in [1/3, \, 1/2  ], \,
\beta  \in [-\sqrt{3}\alpha, \, \sqrt{3} \alpha ].
\]
However, Algorithm~\ref{ag:KKTSDP} successfully
detected nonexistence of GNEs at the loop $k=1$.
\end{example}

\begin{example}\rm
\label{ep:jointsimp}
Consider the rGNEP with jointly simplex constraints
\be\label{eq:jointsimp}
\begin{array}{cllcl}
\min\limits_{x_{1} \in \re^{n_1} }& \frac{x^TA_{1}x+x^Ta_{1}+c_{1}}{x^TB_{1}x+x^Tb_{1}+d_{1}} &\vline&
\min\limits_{x_{2} \in \re^{n_2} }& \frac{x^TA_{2}x+x^Ta_{2}+c_{2}}{x^TB_{2}x+x^Tb_{2}+d_{2}} \\
\st & x_1\in \Delta_1(x_2),&\vline&
\st & x_2\in \Delta_2(x_1).\\
\end{array}
\ee
In the above, for each $i=1,2$,
$A_{i},B_{i}\in\re^{n\times n}$, $a_{i},b_{i}\in\re^n$, $c_{i},d_{i}\in \re$,
and
\[\Delta_i(\xmi)\,\coloneqq\, \left\{x_i\in\re^{n_i} : 1-e^Tx\ge0,\,
x_{i,1}\ge0 \ddd x_{i,n_i}\ge0  \right\}.\]
For both players, we use LMEs as given in (\ref{eq:simplexLME}),
of which denominators have zeros in the feasible set $X=\{x\in\re^n_+: 1-e^Tx\ge0\}$.
Precisely, they vanish when $e^T\xmi = 1,\, i=1,2$.
Moreover, the set of complex KKT points for \reff{eq:jointsimp} has a positive dimension
(see \cite{Nie2022degree}) for all $A_i, B_i, a_i, b_i,c_i$ and $d_i$.
Indeed, for all $t\in[0,1]$,
the pair of $x_1 =(0,t,0\ddd 0)$ and $x_2=(1-t,0\ddd 0)$ is a complex KKT point
because the active constraint gradients
$e, e_2, e_3\ddd e_n$ span the entire space.

For instance, let $n_1=n_2=2$ and
\[
A_1=\lvt\begin{array}{rrrr}
  3 & 2 & -1 & 3\\
  2 & 0 & -2 & 0\\
 -1 &-2 &  0 &-2\\
  3 & 0 & -2 & 2
\end{array}\rvt,\quad
A_2=\lvt\begin{array}{rrrr}
    -1 & 2 & 0 & 0\\
    2 & -2 & 3 & 1\\
    0 & 3 & -4 & 2\\
    0 & 1 & 2 & 2
\end{array}\rvt,
\]
\[
B_1=\lvt\begin{array}{rrrr}
     4 & 0 & 2 & -2\\
     0 & 2 & 0 & -1\\
     2 & 0 & 3 & -1\\
     -2&-1 &-1 & 2
\end{array}\rvt,\quad
B_2=\lvt\begin{array}{rrrr}
     3 & 1 &-1 & 3\\
     1 & 2 &-1 & 2\\
     -1&-1 & 2 & 0\\
     3 & 2 & 0 & 4
\end{array}\rvt,
\]
\[
a_1=\lvt\begin{array}{rrrr}
     1\\1\\-1\\0
\end{array}\rvt,\quad
a_2=\lvt\begin{array}{rrrr}
     -1\\0.5\\1\\-1
\end{array}\rvt,\quad
b_1=\lvt\begin{array}{rrrr}
     0\\-1\\1\\0
\end{array}\rvt,\quad
b_2=\lvt\begin{array}{rrrr}
     1\\0\\-0.5\\1
\end{array}\rvt,
\]
\[
c_1=3,\quad
c_2=-2,\quad
d_1=3.5,\quad
d_2=3.
\]
By a symbolic computation, one can check that the pair of $x_1=(0,t)$ and $x_2=(1-t,0)$
is a KKT point for all $t\in[0,\beta]$,
where $\beta\approx 0.4831$ is the unique real zero of $\beta^3+\frac{1}{3}\beta^2+\frac{97}{48}\beta-\frac{7}{6}=0$.
Apply Algorithm~\ref{ag:KKTSDP} to the GNEP \reff{eq:jointsimp}.
The computational results are displayed in Table~\ref{tab:jointsimp}.
In the $k$th loop, the $u^{(k)} = (u^{(k)}_1,u^{(k)}_2)$
denotes the minimizer of \reff{eq:KKTwithpolyext},
and $\delta^{(k)}$ is the accuracy for $u^{(k)}$ computed as in \reff{eq:dtacc}.
Each feasible extension is selected in form of \reff{eq:simpext}.
\begin{table}[htb]
\centering
\caption{Numerical results for Example~\ref{ep:jointsimp} }
\label{tab:jointsimp}
\begin{tabular}{|c|c|c|c|c|c|c|c|c|c|r|}  \hline
$k$ & \multicolumn{1}{c|}{$(u_1^{(k)}, u_2^{(k)})$}     & $\delta^{(k)}$ \\ \hline
0   & $\quad(0.0000,0.5000),\, (0.0000,0.0000)\quad$  & $-0.1429$  \\ \hline
1   & $(0.0000,0.0000),\, (0.0000,0.0354)$  & $-0.4425$  \\ \hline
2   & $(0.0000,0.4831),\, (0.5169,0.0000)$  & $-0.2476$ \\ \hline
3   & $(0.2910,0.1089),\, (0.6001,0.0000)$  & $-0.0583$ \\ \hline
4   & $(0.0000,0.2742),\, (0.7258,0.0000)$  & $-1.14\cdot10^{-7}$ \\ \hline
\end{tabular}
\end{table}
We get a GNE at the loop $k=4$ with $\delta=-1.14\cdot10^{-7}$.
It took around 16.81 seconds.
\end{example}

\subsection{Some explicit examples}

In the following, we present some explicit examples of rGNEPs.
For cleanness of the paper, we only report computational results
at the last loop for Algorithm~\ref{ag:KKTSDP}.

\begin{example}\rm
\label{ep:firstepinsc4}
(i) Consider the GNEP in (\ref{eq:firsteprtn}).
The LME for the first player is
\[
\lambda_1  \, = \, \Big( \frac{x_{1,2} x_1^T\nabla_{x_1}f_1}{2x_{2,1}}, \, 0, \, 0 \Big) .
 \]
For the second player, the LME is given by (\ref{eq:simplexLME}).
Each LME has a positive denominator on $X$.
Algorithm~\ref{ag:KKTSDP} terminated at the initial loop $k=0$.
The computed GNE is $u = (u_1, u_2)$ with
\begin{equation}
\label{eq:first:gne}
u_1 = (1.3561,0.7374),\quad u_2 = (1.0000,1.0468),\quad \delta=-3.44\cdot10^{-8}.
\end{equation}
It took around 8.36 seconds.

Consider its equivalent polynomials reformulation (\ref{eq:polyreform}).
For the first player, the LME is
\[
\lambda_1 \, = \, (\frac{x_{1,2}}{x_{2,1}}\frac{\pt f_1}{\pt x_{1,2}},\,
0, \, 0, \, x_{1,3}\cdot \frac{\pt f_1}{\pt x_{1,3}}) .
\]
For the second player, the LME is
\[\begin{array}{l}
\lmd_{2,1}=\left(\frac{\pt f_2}{\pt x_{2,1}}-(x_{2,3})^2\frac{\pt f_2}{\pt x_{2,3}}\right)\frac{1-x_{2,1}}{x_{1,1}+x_{1,2}-1}+
\frac{\pt f_2}{\pt x_{2,2}}\frac{1-x_{2,2}}{x_{1,1}+x_{1,2}-1},\\
\lmd_{2,2}=\left(\frac{\pt f_2}{\pt x_{2,1}}-(x_{2,3})^2\frac{\pt f_2}{\pt x_{2,3}}\right)+\lmd_{2,1},\quad
\lmd_{2,3}=\frac{\pt f_2}{\pt x_{2,2}}+\lmd_{2,1},\quad
\lmd_{2,4}=x_{2,3}\cdot \frac{\pt f_2}{\pt x_{2,3}}.
\end{array}\]
Each LME has a positive denominator on $X$.
Algorithm~\ref{ag:KKTSDP} also terminated at the initial loop $k=0$.
The computed GNE is $\hat{u} = (\hat{u}_1, \hat{u}_2)$ with
\[\hat{u}_1 = (1.3561,0.7374,0.7374),\quad \hat{u}_2 = (1.0000,1.0468,1.0000),
\quad \delta=-2.70\cdot10^{-8}\]
The result is consistent with that in \reff{eq:first:gne}.
But the computation took around 264.42 seconds.
It is much more efficient to solve the original rational GNEP.

\noindent
(ii) For the GNEP in (\ref{eq:firsteprtn}),
if objective functions are changed to
\be\label{eq:1stchange}
f_1(x) = \frac{(x_{1,2})^2+x_{1,1}x_{1,2} (e^Tx_2)}{x_{1,1}},\quad
f_2(x) = \frac{(x_{2,2})^2-x_{2,1}x_{2,2} (e^Tx_1)}{x_{2,1}},
\ee
then there is no GNE.
This was detected by Algorithm~\ref{ag:KKTSDP}
at the initial loop $k=0$.
It took about $5.47$ seconds.

Like in (i), we also consider the equivalent polynomial GNEP
with the updated objective.
By applying Algorithm~\ref{ag:KKTSDP}, we detected the
nonexistence of GNEs at the initial loop $k=0$.
It took around 19.61 seconds.

\noindent (iii)
Consider the GNEP in Example~\ref{eq:simpsimplex}.
We use the LMEs as in (\ref{eq:simplexLME})
and the feasible extension as in (\ref{eq:simpext}).
By Algorithm~\ref{ag:KKTSDP},
we got the GNE $u = (u_1, u_2)$ at the loop $k=1$ with
\[
u_1 = (0.0000,0.5000),\quad u_2 = (0.0000,0.5000),\quad \delta=-4.47\cdot10^{-8}.
\]
It took around $3.28$ seconds.

\noindent (iv)
Consider the GNEP in Example~\ref{eq:infKKT}.
We use the LMEs as in (\ref{eq:degenLME}).
Since for each $i$, the feasible set $X_i(\xmi)$ is independent to $\xmi$,
we apply the trivial feasible extension $p_i(x)=x_i$.
By Algorithm~\ref{ag:KKTSDP},
we got the GNE $u = (u_1, u_2)$ in the initial loop with
\[
u_1 = (0.0000,0.0000),\quad u_2 = (1.0000,1.0000),\quad \delta=-5.45\cdot10^{-9}.
\]
It took around $2.03$ seconds.

\noindent (v)
Consider the GNEP in Example~\ref{ep:SDPRP:yalmip}.
For the first player's optimization,
we have the rational LMEs:
\[
\begin{array}{ll}
\lambda_{1,1}=\frac{x_{2,2}-x_{1,1}}{x_{2,2}(2x_{2,1}-x_{1,2})}\cdot \frac{\partial f_1}{\partial x_{1,1}}, &
\lambda_{1,2}=\frac{x_{1,2}x_{2,2}-2x_{1,1}x_{2,1}}{x_{2,1}x_{2,2}(2x_{2,1}-x_{1,2})}\cdot\frac{\partial f_1}{\partial x_{1,1}},\\
\lambda_{1,3}= \frac{2-x_{1,2}x_{2,2}}{3x_{2,2}}\left(
\frac{\partial f_1}{\partial x_{1,2}}+\frac{x_{2,2}-x_{1,1}}{2x_{2,1}-x_{1,2}}\cdot\frac{\partial f_1}{\partial x_{1,1}}
\right), &  \lambda_{1,4}=\frac{1-2x_{1,2}x_{2,2}}{2-x_{1,2}x_{2,2}}\lambda_{1,3}.
\end{array}\]
For the second player's optimization,
we have the rational LMEs:
\[
\begin{array}{ll}
\lambda_{2,1} = \frac{1-x_{2,2}}{2x_{2,1}-1}\cdot\frac{\partial f_2}{\partial x_{2,2}}, &
\lambda_{2,2} = \frac{1-2x_{2,1}x_{2,2}}{2x_{2,1}-1}\cdot\frac{\partial f_2}{\partial x_{2,2}},\\
\lambda_{2,3}=\frac{1}{2}(\lambda_{2,1}-x_{2,1}\frac{\partial f_2}{\partial x_{2,1}}), &
\lambda_{2,4}=\frac{1}{2}\left((2-x_{2,1})\cdot\frac{\partial f_2}{\partial x_{2,1}}+(1-4x_{2,2})\lambda_{2,1}\right).
\end{array}
\]
We apply the feasible extension as in Example~\ref{ep:SDPRP:yalmip}.
Algorithm~\ref{ag:KKTSDP} terminated at the loop $k=1$.
We got the GNE $u=(u_1,u_2)$ with
\[
u_1=(1.0000,0.5000),\quad u_2=(0.5000,1.0000),\quad \delta=-1.82\cdot10^{-8}.
\]
It took around $22.73$ seconds.
\end{example}

\begin{example} \rm
\label{ep:hyperbola}
Consider the $2$-player GNEP with the optimization
\[
\begin{array}{cllcl}
\min\limits_{x_{1} \in \re^3 }& x_1^T(x_1+x_2)+x_{1,1}-x_{1,2}-x_{1,3} &\vline&
\min\limits_{x_{2} \in \re^3 }&  e^Tx_2+\sum_{j=1}^3x_{1,j}(x_{2,j})^2 \\
\st & 1+ (e^Tx_2)^2-x_{1,1}x_{1,2}x_{1,3}\ge0,&\vline&
\st & (e^Tx_1)^2-x_2^Tx_2\ge 0.\\
\end{array}
\]
For the first player's optimization, we have the LME and the feasible extension
\[
\lambda_{1}=-\frac{x_1^T\nabla_{x_1}f_1}{3+ 3(e^Tx_2)^2},\quad
p_{1}(x)=
\left(v_{1,1},v_{1,2}, \frac{1+ (e^Tx_2)^2}{1+ (e^Tu_2)^2}\cdot v_{1,3}\right).
\]
For the second player, we have the LME as in
(\ref{eq:ballLME}) and the feasible extension as in (\ref{eq:annular}).
Algorithm~\ref{ag:KKTSDP} terminated at the loop $k=0$.
We got the GNE $u=(u_1,u_2)$ with
\[
\begin{array}{l}
u_1 = (0.3090,0.8090,0.8090),\quad
u_2 = (-1.6180,-0.6180,-0.6180),
\end{array}
\]
and the accuracy parameter $\dt = -2.77\cdot 10^{-8}$.
It took around $5.16$ seconds.
\end{example}

\begin{example} \rm
\label{ep:3playergame}
\noindent (i)
Consider the $3$-player GNEP
\[
\begin{array}{l}
\mbox{F}_1(x_2,x_3): \,
\left\{\begin{array}{cl}
\min\limits_{x_1\in\re^2} & \|x_1-\frac{1}{2}(x_2+x_3)\|^2\\
\st & x_{1,1}x_{1,2}-x_3^Tx_3-1 = 0,\, x_{1,1}\ge 0,\, x_{1,2}\ge 0,
\end{array}\right.\\[2pt]
\mbox{F}_2(x_1,x_3): \,
\left\{\begin{array}{cl}
\min\limits_{x_2\in\re^2} & x_2^T(x_1+x_3)+(x_{2,1})^3-3(x_{2,2})^2\\
\st & (x_{1,2})^2-\Vert x_{1,1}\cdot x_2\Vert^2 = 0,
\end{array}\right.\\[2pt]
\mbox{F}_3(x_1,x_2): \,
\left\{\begin{array}{cl}
\min\limits_{x_3\in\re^2} & x_3^T(x_1+x_2+x_3-e)\\
\st& x_1^Tx_1-e^Tx_3\ge 0,\ x_{3,1}-0.1\ge0,\ x_{3,2}-0.1\ge0.
\end{array}\right.
\end{array}
\]
The LMEs for $\mbox{F}_1(x_2,x_3)$ and $\mbox{F}_2(x_1,x_3)$ are
\[
\begin{array}{ll}
\lambda_{1,1}=\frac{x_1^T\nabla_{x_1}f_1}{2+2x_3^Tx_3}, &
\lambda_{1,2}=\frac{\partial f_1}{\partial x_{1,1}}-x_{1,2}\lambda_{1,1},\\
\lambda_{1,3}=\frac{\partial f_1}{\partial x_{1,2}}-x_{1,1}\lambda_{1,1}, &
\lambda_{2}= \frac{-x_2^T\nabla_{x_2}f_2}{2(x_{1,2})^2}.
\end{array}
\]
We use the LME as in (\ref{eq:simplexLME}) for $\mbox{F}_3(x_1,x_2)$.
The first two players have the feasible extension
\[
p_1(x)= \Big(v_{1,1},\frac{1+x_3^Tx_3}{v_{1,1}}\Big),\quad
p_2(x)=\frac{u_{1,1}x_{1,2}}{u_{1,2}x_{1,1}}\cdot (v_{2,1},v_{2,2}).
\]
For the third player, the feasible extension is given in (\ref{eq:simpext}).
Algorithm~\ref{ag:KKTSDP} terminated at the initial loop $k=0$.
We got the GNE $u = (u_1, u_2, u_3)$ with
\[
u_1 = (1.1401,1.0461), \quad  u_2 =  (-0.1743,-0.9009), \quad
u_3 = (0.1000,0.4274)
\]
and $\dt = -6.19\cdot 10^{-8}$. It took around $10.58$ seconds.

\noindent (ii)
It is interesting to note that if the third player's objective is changed to
\[
x_3^T(x_1+x_2-e)+(x_{3,1})^2-(x_{3,2})^2,
\]
then there is no GNE. This was detected by
Algorithm~\ref{ag:KKTSDP} at the loop $k=1$.
It took around $19.16$ seconds.
\end{example}

We remark that Algorithm~\ref{ag:KKTSDP} can be generalized
to compute more (or even all) GNEs.
This can be done with the approach in \cite{Nie2020nash}.
Suppose a GNE $u$ is already known.
Select a small scalar $\zeta >0$
and solve the maximization problem
\begin{equation}
\label{eq:max}
\left\{
\begin{array}{rl}
\rho \coloneqq  \max  & [x]_1^T\Theta[x]_1\\
 \st &  x\in \mathscr{U},\,
  [x]_1^T\Theta[x]_1\le [u]_1^T\Theta[u]_1+\zeta.
\end{array}
\right.
\end{equation}
If $\rho> [u]_1^T\Theta[u]_1$,
then let $\zeta \coloneqq \zeta/2$ and solve (\ref{eq:max}) again.
Repeat this until $\zeta$ is small enough to make $\rho = [u]_1^T\Theta[u]_1$.
When $u$ is an isolated KKT point and $\Theta$
is generically positive definite, such $\zeta$ always exists.
This can be proved similarly to that in \cite{Nie2020nash}.
Once such $\zeta$ is found, we add the new inequality
$[x]_1^T\Theta[x]_1\ge [u]_1^T\Theta[u]_1+\zeta$
to (\ref{eq:KKTwithpolyext}). Then Algorithm~\ref{ag:KKTSDP}
can be applied to get a new GNE, if it exists.
It is worth noting that if the optimization
(\ref{eq:KKTwithpolyext}) is infeasible with the newly added constraints,
then there are no other GNEs.
By repeating this process, we can get all GNEs if there are finitely many ones.
We refer to \cite{Nie2020nash} for more details.
The following is such an example.

\begin{example} \rm  \label{ep:rings}
Consider the $2$-player GNEP
\[
\begin{array}{cllcl}
    \min\limits_{x_{1} \in \re^2 }& \frac{x_{2,2}(x_{1,1})^2+
                 x_{2,1} (x_{1,2})^2+x_{1,1}x_{1,2}}{(x_{1,1})^2+1} &\vline&
    \min\limits_{x_{2} \in \re^2 }& \frac{x_{1,2}(x_{2,1})^2+
                 x_{1,1}(x_{2,2})^2+x_{2,1}x_{2,2}}{(x_{2,1})^2+1} \\
    \st & (1-e^Tx_2)^2\le\Vert x_1\Vert^2\le 1 ,&\vline&
    \st &(1-e^Tx_1)^2\le\Vert x_2\Vert^2\le 1.\\
\end{array}
\]
We use the LMEs as in (\ref{eq:annularLME}).
For both $i=1,2$, the feasible extension is
\[
p_i(x)=\frac{v_i}{\Vert v_i\Vert} - \left(\frac{v_i}{\Vert v_i\Vert}-v_i\right)
\frac{e^Tx_{-i}}{e^Tu_{-i}}.
\]
Following the above process, we got two GNEs $u = (u_1, u_2)$ with
\[
\begin{array}{l}
u_1 = (0.9250,-0.3799),\quad u_2 = (0.9250,-0.3799),\quad
        \dt = -9.06\cdot10^{-8},\text{ and} \\
u_1 = (-0.2700,0.9629),\quad u_2 = (-0.2700,0.9629),\quad
     \dt = -2.67\cdot10^{-7}. \\
\end{array}
\]
It took around $29.80$ seconds to get both of them.
Since each rational LME has a positive denominator on $X$,
we obtained all GNEs for this problem.
\end{example}

\subsection{Some examples in applications}

We give some examples arising from applications.
The first one is an NEP with rational objectives.

\begin{example}\rm
\label{ep:electricity}
Consider the NEP for the electricity market problem \cite{Contreras2004,FacKan10}.
Suppose there are $N$ generating companies.
For each $i\in[N]$, the $i$th company possesses $n_i$ generating units,
where the $j$th generating unit has $x_{i,j}$ power generation.
Assume each $x_{i,j}\ge 0$ and is bounded by the maximum capacity $E_{i,j}\ge 0$.
Denote $\varphi_i = (\varphi_{i,1},\ldots, \varphi_{i,n_i})$,
where each $\varphi_{i,j}$ is the cost of the generating unit $x_{i,j}$:
\[
\varphi_{i,j}(x) \, \coloneqq \,
a_{i,j} \cdot (x_{i,j})^3-b_{i,j}\cdot (x_{i,j})^2+c_{i,j}x_{i,j}.
\]
The electricity price is given by $\phi(x) \,  \coloneqq  \, \frac{B}{A+e^Tx}.$
The aim of each company is to maximize its profits.
The $i$th player's optimization problem is
\[
\baray{l}
\mbox{F}_i(x_{-i}):\,
\left\{
    \begin{array}{rl}
    \min  & e^T\varphi_i(x)-
    \phi(x)\cdot e^Tx_i\\
\st  & x_{i,j}\ge0,\quad  E_{i,j}-x_{i,j}\ge0\,(j\in[n_i]).
    \end{array}\right.
\earay
\]
The objectives are rational functions in strategies.
The LME in (\ref{eq:boxLME}) is applicable with box constraints.
Since this is an NEP, we can apply the trivial feasible extension
$p_i(x)=x_i$ for each $i\in[N]$.
We choose the following parameters:
\be \nn
\begin{array}{llllll}  \hline
N=3, & n_1=1,& n_2 = 2, & n_3 = 3, & A=0.5,& B=20,\\
a_{1,1}=0.7, & a_{2,1}=0.75, & a_{2,2}=0.65, & a_{3,1}=0.66, & a_{3,2}=0.7, & a_{3,3}=0.8, \\
b_{1,1}=0.8, & b_{2,1}=0.75, & b_{2,2}=0.65, & b_{3,1}=0.66, & b_{3,2}=0.95,& b_{3,3}=0.5, \\
c_{1,1}=2,   & c_{2,1}=1.25, & c_{2,2}=1,    & c_{3,1}=2.25, & c_{3,2}=3,   & c_{3,3}=3,  \\
E_{1,1}=2,   & E_{2,1}=2.5,  & E_{2,2}=1.5, & E_{3,1}=1.2,  & E_{3,2}=1.8, & E_{3,3}=1.6 .\\
\hline
\end{array}
\ee
Algorithm~\ref{ag:KKTSDP} terminated at the loop $k=0$.
We got the GNE $u = (u_1, u_2, u_3)$, where
\[
u_1 = 1.1432, \quad u_2 = (1.0549,1.1771), \quad
u_3 = (0.8917,0.6439,0.0000),
\]
and $\dt = -1.70\cdot 10^{-8}$.
It took about $7.98$ seconds.
\end{example}

\begin{example}\rm
\label{ep:internet}
Consider the GNEP for internet switching
\cite{facchinei2009generalized,Kesselman2005}.
Assume there are $N$ users, and the maximum capacity of the buffer is $B$.
Let $x_i$ denote the amount of $i$th user's ``packets" in the buffer,
which has a positive lower bound $L_i$.
Suppose the buffer is managed with \textit{``drop-tail" policy}:
if the buffer is full, further packets will be lost and resent.
Suppose $\frac{x_i}{e^Tx}$ is the transmission rate of the $i$th user,
$\frac{e^Tx}{B}$ is the congestion level of the buffer,
and $1-\frac{e^Tx}{B}$ measures the decrease
in the utility of the $i$th user as the congestion level increases.
The $i$th user's optimization problem is
\begin{equation}
\label{eq:internetrtn}
\left\{ \begin{array}{cl}
   \min\limits_{x_i \in \re^1 } &
    f_i(x)=-\frac{x_i}{e^Tx}(1-\frac{e^Tx}{B}) \\
 \st & x_i- L_i \ge 0, \, B-e^Tx \ge 0.
\end{array} \right.
\end{equation}
We apply the LME as in (\ref{eq:simplexLME})
and solve the GNEP for $N=10,\dots,14$,
with parameters $B=2.5$ and $L_i=0.09+0.01i$ for each $i\in[N]$.
Algorithm \ref{ag:KKTSDP} terminated at the initial loop $k=0$ for each case.
The numerical results are shown in Table~\ref{tab:internet}.
In the table, $u=(u_1, \ldots, u_N)$ and $\delta$
denote respectively the GNE and the accuracy parameter,
and ``time'' is the CPU time in seconds.
\begin{table}[htb]
\centering
\caption{Numerical results of Example~\ref{ep:internet} }
\label{tab:internet}
\begin{tabular}{|c|c|c|c|c|c|c|c|c|c|r|}  \hline
$N$ & \multicolumn{1}{c|}{$u=(u_1,\ldots,u_N)$}     & $\delta$ & time \\ \hline
10   & $\ \, u_i = 0.2250\ (i=1,\ldots,10)$  & $-1.05\cdot10^{-9}$ & 11.16 \\ \hline
11   & $\ \, u_i = 0.2066\ (i=1,\dots,11)$  & $-4.75\cdot10^{-9}$ & 24.36 \\ \hline
12   & $u_i = \left\{\begin{array}{ll}0.1883 & (i=1,\ldots,9)\\
      L_i & (i=10,\ldots,12)\end{array}\right.$  & $-1.93\cdot10^{-8}$ & 45.38 \\ \hline
13   & $u_i = \left\{\begin{array}{ll}0.1647 & (i=1,\ldots,7)\\
      L_i & (i=8,\ldots,13)\end{array}\right.$  & $-4.83\cdot10^{-8}$ & 70.81 \\ \hline
14   & $u_i = \left\{\begin{array}{ll}0.1282 & (i=1,2,3)\\
      L_i & (i=4,\dots,14)\end{array}\right.$  & $-1.02\cdot10^{-7}$ & 97.00 \\ \hline
\end{tabular}
\end{table}
\end{example}

\subsection{Comparison with other methods}
\label{ssc:compare}

We compare our method (i.e., Algorithm~\ref{ag:KKTSDP})
with some existing methods for solving GNEPs, such as
the interior point method (IPM)
based on the KKT system \cite{dreves2011solution},
the quasi-variational inequality method (QVI) in \cite{Han2012},
the Augmented-Lagrangian method (ALM) in \cite{kanzow2016},
and the Gauss-Seidel method (GSM) in \cite{Nie2020gs}.
For Example~\ref{ep:rings}, we only compare for finding one GNE.
For Example~\ref{ep:internet}, we compare for $N=10$.

For a computed tuple $u \coloneqq (u_1, \ldots, u_N)$, we use the quantity
\[
\kappa \coloneqq \max \Bigg
 \{\max_{i\in[N],  j\in \mc{I}^{(i)}_1 \cup  \mc{I}^{(i)}_2 } \{-g_{i,j}(u)\},
\max_{i\in[N],j\in\mc{I}^{(i)}_0 }\{|g_{i,j}(u)|\} \Bigg\}
\]
to measure the feasibility violation.
Note that $u$ is feasible if and only if $\kappa \le 0$ and
$g_{i,j}(u)>0$ for every $j \in \mc{I}^{(i)}_2$.
For these methods, we use the following stopping criterion:
for each generated iterate $u$,
if its feasibility violation $\kappa <10^{-6}$,
then we compute the accuracy parameter $\dt$ for verifying GNEs.
If $\delta > -10^{-6}$, then we stop the iteration.

For the above methods, the parameters are the same as in
\cite{dreves2011solution,kanzow2016,Nie2020gs}.
The full penalization is used for the Augmented-Lagrangian method,
and a Levenberg-Marquardt type method (see \cite[Algorithm~24]{kanzow2016})
is used to solve penalized subproblems.
For the Gauss-Seidel method,
the normalization parameters are updated as (4.3) in \cite{Nie2020gs},
and the Moment-SOS relaxations are used to
solve each player's optimization problems.
For the QVI method, the Moment-SOS relaxations are used to compute projections.
We let $1000$ be the maximum number of iterations
for all the above methods.
For initial points, we use $(0,1,1,0)$ for Examples~\ref{ep:quad1}(i)-(ii),
$(1,1,1,1)$ for Examples~\ref{ep:firstepinsc4}(i)(ii)(iv)(v),
$(\sqrt{2},\sqrt{2},1,1,1,1)$ for Example~\ref{ep:3playergame},
$(0,1,0,1)$ for Example~\ref{ep:rings},
$0.25\cdot (1,\cdots,1)$ for Example~\ref{ep:internet},
and the zero vectors for other examples.
If the maximum number of iterations is reached
but the stopping criterion is not met,
we still solve (\ref{eq:checkopt:alg}) to check if
the latest iterating point is a GNE or not.
For the QVI, the produced sequence
is said to converge if the projection residue is sufficiently small.
For the ALM and IPM, the produced sequence is
considered to converge if the last iterate satisfies the KKT
conditions up to a small round-off error (say, $10^{-6}$).
The numerical results are shown in Table~\ref{tab:comparison}.
The ``$u$" column lists the most recent update by each method,
``time" gives the {total CPU time (in seconds)},
and the ``{$\max\{|\delta|,\kappa\}$}"
measures the feasibility violation and the accuracy of being GNEs.
For all methods in the table,
if the produced sequence is convergent, but the quantity
$\max\{|\delta|,\kappa\}$ is not close to zero (e.g., $\le 10^{-6}$),
then the method converges to a KKT point that is not a GNE.

\begin{longtable}{|l|c|c|c|c|c|c|c|r|c|}
\caption{Comparison with some existing methods}\label{tab:comparison}\\
   \hline
   Algorithm    &   $u$   &  time    &  {$\max\{|\delta|,\kappa\}$}  \\ \hline
   \endfirsthead
   \multicolumn{4}{|c|}{\bf Example~\ref{ep:quad1}(i)} \\ \hline
   ALM                 &  \multicolumn{3}{c|}{not convergent }\\ \hline
   IPM                 &  \multicolumn{3}{c|}{not convergent } \\ \hline
   QVI                 &  (0.8911,-0.0000,0.8910,0.0000) & $298.10$ &  $0.22$ \\ \hline
   GSM                 &  (0.4930,-0.0835,0.5000,0.4930) & $3.12$  & $1.33\cdot10^{-8}$ \\ \hline
   Alg.~\ref{ag:KKTSDP}  & (0.4930,-0.0835,0.5000,0.4930)  & $4.22$  & $4.31\cdot10^{-7}$ \\ \hline
   \multicolumn{4}{|c|}{\bf Example~\ref{ep:quad1}(ii)} \\ \hline
   ALM                 &  (0.5000,0.8660,1.0000,0.0000) & $63.81$ &  $2.25$\\ \hline
   IPM                 &  \multicolumn{3}{c|}{not convergent } \\ \hline
   QVI                 &  \multicolumn{3}{c|}{not convergent } \\ \hline
   GSM                 &  \multicolumn{3}{c|}{not convergent } \\ \hline
   Alg.~\ref{ag:KKTSDP}  & nonexistence of GNEs detected  & $5.56$ &  \\ \hline
   \multicolumn{4}{|c|}{\bf Example~\ref{ep:jointsimp}} \\ \hline
   ALM                 &  (0.0000,0.1931,0.2889,0.0000) & 47.51 & $0.21$\\ \hline
   IPM                 &  (0.0000,0.1931,0.2889,0.0000) & 17.00 &  $0.21$ \\ \hline
   QVI                 &  (0.0000,0.0000,0.0000,0.0354) & 441.52 &  $0.44$ \\ \hline
   GSM                 &  (0.0000,0.0000,1.0000,0.0000) & 0.59  & $8.08\cdot10^{-8}$ \\ \hline
   Alg.~\ref{ag:KKTSDP}  & (0.0000,0.2742,0.7258,0.0000)  & 16.81  & $1.14\cdot10^{-7}$ \\ \hline

   \multicolumn{4}{|c|}{\bf Example~\ref{ep:firstepinsc4}(i)} \\ \hline
   ALM      		   & \multicolumn{3}{c|}{not convergent } \\ \hline
   IPM      		   &  (1.3561,0.7374,1.0000,1.0468) & 2.39 &  $1.93\cdot 10^{-7}$ \\ \hline
   QVI                 &  (1.3562,0.7375,1.0000,1.0469) & 2753.26 &  $1.34\cdot 10^{-4}$ \\ \hline
   GSM      		   & (1.3558,0.7376,1.0000,1.0466)          & 3.47  & $2.60\cdot10^{-9}$ \\ \hline
   Alg.~\ref{ag:KKTSDP}  & (1.3561,0.7374,1.0000,1.0468)    	   & 8.36  & $3.44\cdot10^{-8}$           \\ \hline

   \multicolumn{4}{|c|}{\bf Example~\ref{ep:firstepinsc4}(ii)} \\ \hline
   ALM      		   & \multicolumn{3}{c|}{not convergent} \\ \hline
   IPM      		   & \multicolumn{3}{c|}{not convergent} \\ \hline
   QVI                 & \multicolumn{3}{c|}{not convergent} \\ \hline
   GSM      		   & \multicolumn{3}{c|}{not convergent} \\ \hline
   Algorithm~3.3  & nonexistence of GNEs detected    	   & 5.47  &             \\ \hline

  \multicolumn{4}{|c|}{\bf Example~\ref{ep:firstepinsc4} (iii)} \\ \hline
   ALM      		   & (0,0,0,0) & 49.34 & $1.00$ \\ \hline
   IPM      		   & (0.2808,0.2192,0.2808,0.2192) & 12.98 & $0.16$ \\ \hline
   QVI                 & (0.0000,0.4999,0.0001,0.4999) & 616.29 & $5.35\cdot 10^{-5}$ \\ \hline
   GSM     		   & (0.0000,0.4995,0.0000,0.4995) & 110.79 & $8.58\cdot 10^{-4}$  \\ \hline
   Alg.~\ref{ag:KKTSDP}  & (0.0000,0.5000,0.0000,0.5000)    	   & 3.28  &  $4.47\cdot 10^{-8}$           \\ \hline

    \multicolumn{4}{|c|}{\bf Example~\ref{ep:firstepinsc4}(iv)} \\ \hline
   ALM                 & \multicolumn{3}{c|}{not convergent} \\ \hline
   IPM                 & \multicolumn{3}{c|}{not convergent} \\ \hline
   QVI                 & \multicolumn{3}{c|}{not convergent} \\ \hline
   GSM                 & \multicolumn{3}{c|}{not convergent} \\ \hline
   Alg.~\ref{ag:KKTSDP}  & (0.0000,0.0000,1.0000,1.0000)           & $2.03$  &  $5.45\cdot10^{-9}$ \\ \hline

   \multicolumn{4}{|c|}{\bf Example~\ref{ep:firstepinsc4}(v)} \\ \hline
   ALM      		   & \multicolumn{3}{c|}{not convergent} \\ \hline
   IPM      		   & \multicolumn{3}{c|}{not convergent} \\ \hline
   QVI                 & (1.0000,0.5000,0.5000,1.0000) & 490.93 & $9.51\cdot10^{-5}$ \\ \hline
   GSM      		   & (1.0000,0.5000,0.5000,1.0000) & 1.80 & $2.31\cdot10^{-10}$ \\ \hline
   Alg.~\ref{ag:KKTSDP}  & (1.0000,0.5000,0.5000,1.0000)    	   & 22.73  &  $1.82\cdot10^{-8}$ \\ \hline

   \multicolumn{4}{|c|}{\bf Example~\ref{ep:hyperbola}} \\ \hline
   ALM      		   & \multicolumn{3}{c|}{not convergent} \\ \hline
   IPM      		   & \multicolumn{3}{c|}{not convergent} \\ \hline
   QVI                 & $\begin{array}{l}(0.3094,0.8090,0.8090,\\\qquad-1.6172,-0.6180,-0.6180) \end{array}$ & 21.46  &  $5.63\cdot 10^{-7}$           \\ \hline
   GSM     		   & \multicolumn{3}{c|}{not convergent} \\ \hline
   Alg.~\ref{ag:KKTSDP}  & $\begin{array}{l}(0.3090,0.8090,0.8090,\\\qquad-1.6180,-0.6180,-0.6180) \end{array}$ & 5.16  &  $2.77\cdot 10^{-8}$           \\ \hline

   \multicolumn{4}{|c|}{\bf Example~\ref{ep:3playergame}(i)} \\ \hline
   ALM      		   & $\begin{array}{l}(0.7774,1.3629,-0.2227,\\\qquad1.7389,0.2226,0.1000) \end{array}$ & 75.92  &  $5.10$ \\ \hline
   IPM      		   & $\begin{array}{l}(1.1401,1.0461,-0.1743,\\\qquad-0.9009,0.1000,0.4274) \end{array}$ & 0.86  &  $8.24\cdot 10^{-7}$\\ \hline
   QVI                 & $\begin{array}{l}(0.7775,1.3628,-0.2227,\\\qquad1.7386,0.2227,0.1000) \end{array}$ & 192.73  &  $5.10$\\ \hline
   GSM     		   & $\begin{array}{l}(1.1403,1.0463,-0.1743,\\\qquad-0.9009,0.1000,0.4273) \end{array}$ & 6.28  &  $1.88\cdot 10^{-8}$\\ \hline
   Alg.~\ref{ag:KKTSDP}  & $\begin{array}{l}(1.1401,1.0461,-0.1743 \\ \qquad-0.9009,0.1000,0.4274) \end{array}$ & 10.58  &  $6.19\cdot 10^{-8}$           \\ \hline

   \multicolumn{4}{|c|}{\bf Example~\ref{ep:3playergame}(ii)} \\ \hline
   ALM      		   & \multicolumn{3}{c|}{not convergent} \\ \hline
   IPM      		   & \multicolumn{3}{c|}{not convergent} \\ \hline
   QVI                 & \multicolumn{3}{c|}{not convergent} \\ \hline
   GSM     		   & \multicolumn{3}{c|}{not convergent} \\ \hline
   Alg.~\ref{ag:KKTSDP}  & nonexistence of GNEs detected    	   & 19.16  &             \\ \hline

   \multicolumn{4}{|c|}{\bf Example~\ref{ep:rings}} \\ \hline
   ALM               &\multicolumn{3}{c|}{not convergent}           \\ \hline
   IPM               & $(0.2665, 0.3184, 0.2665, 0.3184)$ & 11.22  &  $0.27$  \\ \hline
   QVI                 & \multicolumn{3}{c|}{not convergent} \\ \hline
   GSM                & \multicolumn{3}{c|}{not convergent} \\ \hline
   Alg.~\ref{ag:KKTSDP}  &  $(0.9250,-0.3799,0.9250,-0.3799)$ & 2.78  &$9.06\cdot 10^{-8}$           \\ \hline

   \multicolumn{4}{|c|}{\bf Example~\ref{ep:electricity}} \\ \hline
   ALM      		   & $\begin{array}{l} (1.1652,1.0601, 1.1822,\\\qquad 0.9952, 0.0577,0.2332) \end{array}$ & 94.36  &  $0.10$           \\ \hline
   IPM      		   & \multicolumn{3}{c|}{not convergent} \\ \hline
   QVI               & $\begin{array}{l} (1.1432,1.0549, 1.1770,\\\qquad 0.8916, 0.6440,0.0001) \end{array}$ & 523.06  &  $2.35\cdot 10^{-5}$           \\ \hline
   GSM      		   & $\begin{array}{l} (1.1446,1.0551, 1.1772,\\\qquad 0.8917, 0.6431,0.0000) \end{array}$ & 4.22  &  $9.16\cdot 10^{-7}$           \\ \hline
   Alg.~\ref{ag:KKTSDP}  & $\begin{array}{l}(1.1432,1.0549, 1.1771,\\\qquad 0.8917, 0.6439, 0.0000) \end{array}$ & 7.98  &  $1.70\cdot 10^{-8}$           \\ \hline

   \multicolumn{4}{|c|}{\bf Example~\ref{ep:internet}} \\ \hline
   ALM      		   & $\begin{array}{l} (0.2250,0.2250, 0.2250, 0.2250,\\\qquad 0.2250, 0.2250, 0.2250\\\qquad\qquad 0.2250, 0.2250, 0.2250) \end{array}$ & 3.06 &  $5.28\cdot 10^{-12}$           \\ \hline
   IPM      		   & $\begin{array}{l} (0.2245,0.2245, 0.2246, 0.2246,\\\qquad 0.2246, 0. 2246, 0.2247\\\qquad\qquad 0.2251, 0.2260, 0.2275) \end{array}$ & 10.89  &  $5.13\cdot 10^{-7}$  \\ \hline
   QVI                & $\begin{array}{l} (0.2254,0.2254, 0.2254, 0.2254,\\\qquad 0.2254, 0.2253, 0.2253\\\qquad\qquad 0.2253, 0.2252, 0.2251) \end{array}$ & 9.10  &  $4.59\cdot 10^{-7}$  \\ \hline
   GSM      		   & $\begin{array}{l} (0.2236,0.2250, 0.2262, 0.2270,\\\qquad 0.2271, 0.2266, 0.2256\\\qquad\qquad 0.2245, 0.2236, 0.2232) \end{array}$ & 21.33  &  $8.59\cdot 10^{-7}$  \\ \hline
   Alg.~\ref{ag:KKTSDP}   & $\begin{array}{l} (0.2250,0.2250, 0.2250, 0.2250,\\\qquad 0.2250, 0.2250, 0.2250\\\qquad\qquad 0.2250, 0.2250, 0.2250) \end{array}$ & 11.16 &  $1.05\cdot 10^{-9}$       \\ \hline
\end{longtable}

The comparisons are summarized as follows:
\bit

\item
The ALM failed to get a GNE for Examples~\ref{ep:firstepinsc4}(i),(ii),(iv),
\ref{ep:hyperbola} and \ref{ep:3playergame}(ii),
because the penalization subproblems could not be solved accurately.
It converged to non-GNE KKT points for Examples~\ref{ep:quad1}(ii), \ref{ep:jointsimp}, \ref{ep:firstepinsc4}(iii),
\ref{ep:3playergame}(i) and \ref{ep:electricity}.
It did not converge for Examples~\ref{ep:quad1}(i),
\ref{ep:firstepinsc4}(v) and \ref{ep:rings},
when the maximum penalty parameter $10^{12}$ was reached.

\item
The IPM failed to get a GNE for
Examples~\ref{ep:firstepinsc4}(iv), \ref{ep:hyperbola} and \ref{ep:3playergame}(ii),
because the step length was too small to
efficiently decrease the violation of KKT conditions.
It converged to non-GNE KKT points for Examples~\ref{ep:quad1}(i)-(ii),
\ref{ep:firstepinsc4}(iii) and \ref{ep:electricity}.
It did not converge for Examples~\ref{ep:quad1}(i)-(ii),
\ref{ep:firstepinsc4}(ii),(v) and \ref{ep:electricity},
because the Newton type directions
did not satisfy sufficient descent conditions.

\item
The QVI converged to non-GNE points for Examples~\ref{ep:quad1}(i), \ref{ep:jointsimp},~\ref{ep:3playergame}(i).
It did not converge for
Examples~\ref{ep:quad1}(ii), \ref{ep:firstepinsc4}(ii),(iv),
\ref{ep:3playergame}(ii) and \ref{ep:rings},
since the projection could not be computed successfully.

\item
The GSM failed to find a GNE for
Examples~\ref{ep:quad1}(ii),~\ref{ep:firstepinsc4}(ii),(iv),
\ref{ep:hyperbola}, \ref{ep:3playergame}(ii) and \ref{ep:rings},
because some sub-optimization problems could not be solved successfully.
It terminated at the maximum iteration number for
Example~\ref{ep:firstepinsc4}(iii), but did not meet
the stopping criterion.

\eit

\subsection{About strict inequality constraints}
\label{ssc:strict}

For rGNEPs, rational Lagrange multiplier expressions are used to get the KKT set.
For strict inequality constraints, their Lagrange multipliers are always zeros.
In Algorithm~\ref{ag:KKTSDP}, the set $\mc{K}$ is as in \reff{eq:KKTweak},
where the LMEs are zeros for strict inequalities.
For each rational optimization problem, its feasible set
is relaxed from \reff{rat:feasC} to \reff{def:K_1},
and then we solve it by Algorithm~\ref{ag:ratopt}.
Strict inequalities give open sets.
When there are finitely many KKT points (this is the generic case),
there does not exist a sequence of feasible KKT points
that converge to the boundary given by strict inequality constraints.
For some special cases, the KKT set may be infinite
and there possibly exists a sequence of feasible KKT points
converging to the boundary of strict inequality constraints.
If this case happens, the limit may not be a GNE.
The following is such an example.

\begin{example}\rm
\label{ep:boundary}
Consider the following GNEP
\be
\label{eq:boundary}
\begin{array}{rllrl}
    \min\limits_{x_{1} \in \re^1 }& x_1x_2&\vline&
    \min\limits_{x_{2} \in \re^1 }& \frac{-(x_2)^2}{1-(x_1)^2} \\
    \st &  x_1 \ge 0,\, 1-x_1 \ge 0, &\vline&
     \st & x_2\ge 0,\, 1-(x_1)^2-(x_2)^2 > 0.
\end{array}
\ee
The second player has a strict inequality constraint.
The Lagrange multiplier vectors can be expressed as
\[
\lambda_1 = ( x_2-x_1x_2, \quad  -x_1x_2 ),\qquad
\lambda_2 = \Big(\frac{-2x_2}{1-(x_1)^2}, \quad 0\Big).
\]
The denominators of $\lambda_2$ and the second player's objective
are positive in the feasible set, but not positive on
the boundary of its closure.
The KKT set $\mc{K}$ is
\[
\mc{K} =
\left\{ (x_1, x_2)
\left|\begin{array}{c}
  x_1(x_2-x_1x_2) =0, \,  -x_1x_2(1-x_1) = 0, \\
0\le x_1\le 1,\, x_2-x_1x_2\ge 0,\, -x_1x_2\ge 0,\\
x_2 \cdot \frac{-2x_2}{1-(x_1)^2} = 0, \\
x_2\ge 0,\, (x_1)^2+(x_2)^2<1,\, \frac{-2x_2}{1-(x_1)^2} \ge 0.
\end{array}\right.
\right \}.
\]
One can see that $\mc{K} = \{ 0 \le x_1 < 1, x_2 = 0 \}$.
After the cancellation for the denominator
and relaxing $(x_1)^2+(x_2)^2<1$ to the weak inequality
$(x_1)^2+(x_2)^2 \le 1$, the set $\mc{K}$ is changed to
\[
\mc{K}_1 =
\left\{ (x_1, x_2)
\left|\begin{array}{c}
  x_1(x_2-x_1x_2) =0, \,  -x_1x_2(1-x_1) = 0, \\
0\le x_1\le 1,\, x_2-x_1x_2\ge 0,\, -x_1x_2\ge 0,\\
x_2 \cdot (-2x_2) = 0, \\
x_2\ge 0,\, (x_1)^2+(x_2)^2 \le 1,\,  -2x_2  \ge 0.
\end{array}\right.
\right \}.
\]
Then one can check that $\mc{K}_1 = \{ 0 \le x_1 \le 1, x_2 = 0 \}$, i.e.,
\[
\mc{K}_1 = [0,1]\times \{0\},\quad\mbox{and}\quad
\mc{K}_1\setminus\mc{K} = \{(1,0)\}.
\]
When we apply the algorithm to compute GNEs.
We got the candidate
$\hat{x} = (1,0)$, which is not feasible for \reff{eq:boundary}
but lies on the boundary.
The second player's objective is not well defined at $\hat{x}$.
The candidate $\hat{x} = (1,0)$ is not a GNE.
Indeed, this GNEP does not have any GNE.
\end{example}

\section{Conclusions and Discussions}
\label{sc:conc}

This paper studies how to solve GNEPs given by rational functions.
Lagrange multiplier expressions and feasible extensions
are introduced to compute GNEs.
We propose a hierarchy of rational optimization problems
to solve GNEPs. This is given in Algorithm~\ref{ag:KKTSDP}.
The Moment-SOS relaxations are used to
solve the appearing rational optimization problems.
Under some general assumptions,
we show that Algorithm~\ref{ag:KKTSDP} can get a GNE if it exists
or detect its nonexistence.

The feasible extension is a major technique used in this paper.
Its purpose is to preclude KKT points that are not GNEs.
This technique was originally introduced for
solving bilevel optimization in the work \cite{Nie2020bilevel}.
However, their properties are quite different
for GNEPs and bilevel optimization.
For instance, a generic polynomial GNEP has finitely many KKT points,
which is implied by the recent work \cite[Theorem~3.1]{Nie2022degree}.
It guarantees the existence of feasible extensions for generic rGNEPs,
which is shown in Theorem~\ref{thm:finitK}.
So, Algorithm~\ref{ag:KKTSDP} has finite convergence for general cases.
However, for general polynomial bilevel optimization,
the KKT set (for the lower level optimization) is usually not finite.
There do not exist results on the existence of feasible extensions.
Moreover, the work \cite{Nie2020bilevel} only considers polynomial extensions.
In this paper, we consider more general feasible extensions
that are given by rational functions.
It greatly broadens the usage of feasible extensions for solving GNEPs.
For instance, we gave explicit rational feasible extensions in \reff{eq:annular}
for ball constraints parameterized by the polynomial $a_j(x_{-i})$.
For this kind of constraints, polynomial extensions
as in \cite{Nie2020bilevel} usually do not exist.

There exists much interesting future work to do with feasible extensions.
For instance, are there sufficient conditions
weaker than those in Theorem~\ref{thm:finitK}
for the existence of feasible extensions?
If they exist, how can we find them efficiently?
These questions are mostly open.

\medskip \noindent
{\bf Acknowledgement}
The first author is partially supported by the NSF grant
DMS-2110780.


\begin{thebibliography}{100}


\bibitem{ardagna2017generalized}
D.~Ardagna, M.~Ciavotta and M.~Passacantando,
Generalized Nash equilibria for the service provisioning problem in multi-cloud systems,
\emph{IEEE Transactions on Services Computing}
10, 381--395, 2017.


\bibitem{Ba2020}
Q.~Ba and J.~Pang,
Exact penalization of generalized Nash equilibrium problems,
\emph{Operations Research}
70(2), 1448--1464, 2022.

\bibitem{Brks}
D.~Bertsekas,
\emph{Nonlinear Programming}, third edition,
Athena Scientific, 2016.



\bibitem{Borgens2021}
E.~B\"{o}rgens and C.~Kanzow,
\newblock ADMM-type methods for generalized Nash equilibrium problems in Hilbert Spaces,
\newblock \emph{SIAM Journal on Optimization}
31(1), 377--403, 2021.



\bibitem{chen2020oil}
X.~Chen, Y.~Shi and X.~ Wang,
Equilibrium oil market share under the COVID-19 pandemic,
\emph{Preprint}, 2020. \url{arXiv:2007.15265}


\bibitem{Clarke}
F.~Clarke,
\emph{Optimization and Nonsmooth Analysis,}
{Society for Industrial and Applied Mathematics}, 1990.


\bibitem{Contreras2004}
J.~Contreras, M.~Klusch and J.B.~Krawczyk,
Numerical solutions to Nash-Cournot equilibria in coupled constraint electricity markets,
\emph{IEEE Transactions on Power Systems}
19(1), 195--206, 2004.


\bibitem{Cui2021book}
Y.~Cui and J.~Pang,
\emph{Modern nonconvex nondifferentiable optimization},
Society for Industrial and Applied Mathematics, 2021.


\bibitem{dreves2011solution}
A.~Dreves, F.~Facchinei, C.~Kanzow, and S.~Sagratella,
On the solution of the KKT conditions of generalized Nash equilibrium problems,
\emph{SIAM Journal on Optimization}
21, 1082--1108, 2011.

\bibitem{dreves2012nonsmooth}
A.~Dreves, C.~Kanzow, and O.~Stein,
Nonsmooth optimization reformulations of
player convex generalized Nash equilibrium problems,
\emph{Journal of Global Optimization}
53(4), 587--614, 2012.


\bibitem{Facchinei2010generalized}
F.~Facchinei, A.~Fischer, and V.~Piccialli,
On generalized Nash games and variational inequalities,
\emph{Operations Research Letters}
35(2), 159--164, 2007.

\bibitem{facchinei2009generalized}
F.~Facchinei, A.~Fischer, and V.~Piccialli,
Generalized Nash equilibrium problems and Newton methods,
\emph{Mathematical Programming}
117(1), 163--194, 2009.


\bibitem{Facchinei2010}
F.~Facchinei and C.~Kanzow,
Generalized Nash equilibrium problems,
\emph{Annals of Operations Research}
175(1), 177--211, 2010.

\bibitem{FacKan10}
F.~Facchinei and C.~Kanzow,
Penalty methods for the solution of generalized Nash equilibrium problems,
\emph{SIAM Journal on Optimization}
20(5), 2228--2253, 2010.

\bibitem{Facchinei2010book}
F.~Facchinei and J.~Pang,
Nash equilibria: the variational approach,
\emph{Convex optimization in signal processing and communications},
D.~Palomar, Y.~Eldar, eds., 443--493,
Cambridge University Press, England, 2010.


\bibitem{Facchinei2011}
F.~Facchinei, V.~Piccialli, and M.~Sciandrone,
Decomposition algorithms for generalized potential games,
\emph{Computational Optimization and Applications}
50(2), 237--262, 2011.



\bibitem{fischer2014generalized}
A.~Fischer, M.~Herrich, and K.~Sch{\"o}nefeld,
Generalized Nash equilibrium problems-recent advances and challenges,
\emph{Pesquisa Operacional}
34(3), 521--558, 2014.



\bibitem{GuoLinYeZhang}
L.~Guo, G.~Lin, J.J.~Ye and J.~Zhang,
Sensitivity analysis of the value function for
parametric mathematical programs with equilibrium constraints,
\emph{SIAM Journal on Optimization}
24(3), 1206---1237, 2014.

\bibitem{Han2012}
D.~Han, H.~Zhang, G.~Qian and L.~Xu,
An improved two-step method for solving generalized Nash equilibrium problems,
\emph{European Journal of Operational Research}
216(3), 613--623, 2012.


\bibitem{HenLas05}
D.~Henrion and J.~Lasserre,
Detecting global optimality and extracting solutions in GloptiPoly,
\emph{Positive polynomials in control},
Springer, Berlin, Heidelberg, 293--310, 2005.

\bibitem{GloPol3}
D.~Henrion, J.~Lasserre, and J.~L\"{o}fberg,
Gloptipoly 3: moments, optimization and semidefinite programming,
\emph{Optimization Methods and Software}
24(4-5), 761--779, 2009.



\bibitem{HilNie08}
C.~Hillar and J.~Nie,
\newblock An elementary and constructive solution to
Hilbert’s 17th problem for matrices,
\newblock \emph{Proceedings of the AMS}
136(1), 73--76, 2008.



\bibitem{JdeK06}
D.~Jibetean and E.~de~Klerk.
Global optimization of rational functions:
A semidefinite programming approach,
\emph{Mathematical Programming}
106(1), 93--109, 2006.


\bibitem{kanzow2016}
C.~Kanzow and D.~Steck,
Augmented Lagrangian methods for the solution of generalized Nash equilibrium problems,
\emph{SIAM Journal on Optimization}
26(4), 2034--2058, 2016.


\bibitem{Kesselman2005}
A.~Kesselman, S.~Leonardi, and V.~Bonifaci.
Game-theoretic analysis of internet switching with selfish users,
\emph{International Workshop on Internet and Network Economics},
(pp. 236-245), Springer, Berlin, Heidelberg, 2005.



\bibitem{Las01}
J.~Lasserre,
Global optimization with polynomials and the problem of moments,
\emph{SIAM Journal on Optimization}
11(3), 796--817, 2001.



\bibitem{Las10CovRep}
J.~Lasserre,
On representations of the feasible set in convex optimization,
\emph{Optimization Letters}
4(1), 1-5, 2010.

\bibitem{Las15}
J.~Lasserre,
\emph{An Introduction to Polynomial and Semi-Algebraic Optimization},
Cambridge University Press, Volume 52, 2015.



\bibitem{lasserre2008semidef}
J.~Lasserre, M.~Laurent and P.~Rostalski,
Semidefinite characterization and computation of zero-dimensional real radical ideals,
\emph{Foundations of Computational Mathematics}
8(5), 607--647, 2008.


\bibitem{Lau09}
M.~Laurent,
Sums of squares, moment matrices and optimization over polynomials,
\emph{Emerging Applications of Algebraic Geometry of IMA Volumes in Mathematics and its Applications}, vol.~149, pp. 157-270, Springer, 2009.

\bibitem{Liu2016}
M.~Liu, and O.~Tuzel,
Coupled generative adversarial networks,
\textit{Advances in neural information processing systems},
29 (2016): 469-477.


\bibitem{nabetani2011parametrized}
K.~Nabetani, P.~Tseng, and M.~Fukushima,
Parametrized variational inequality approaches to
generalized Nash equilibrium problems with shared constraints,
\emph{Computational Optimization and Applications}
48, 423--452, 2011.


\bibitem{NDG08}
J.~Nie, J.~Demmel and M.~Gu,
Global minimization of rational functions and the nearest GCDs.
\emph{Journal of Global Optimization}
40(4),  697--718, 2008.


\bibitem{nie2013certifying}
J.~Nie,
Certifying convergence of Lasserre's hierarchy via flat truncation,
\emph{Mathematical Programming}
142(1-2), 485--510, 2013.


\bibitem{nie2015linear}
J.~Nie,
Linear optimization with cones of moments and nonnegative polynomials,
\emph{Mathematical Programming}
153(1), 247--274, 2013.


\bibitem{Nie2019}
J.~Nie,
Tight relaxations for polynomial optimization and Lagrange multiplier expressions,
\emph{Mathematical Programming}
178(1-2), 1--37, 2019.


\bibitem{PMI2011}
J.~Nie,
\newblock Polynomial matrix inequality and semidefinite representation,
\newblock \emph{Mathematics of Operations Research}
36(3), 398--415, 2011.


\bibitem{Nie2012}
J.~Nie,
 Sum of squares methods for minimizing polynomial forms
over spheres and hypersurfaces,
\emph{Frontiers of mathematics in china}
7(2), 321--346, 2012.


\bibitem{NieZhang18}
J.~Nie and X.~Zhang,
Real eigenvalues of nonsymmetric tensors,
\emph{Computational Optimization and Applications}
70(1), 1--32, 2018.




\bibitem{Nie2020gs}
J.~Nie, X.~Tang and L.~Xu,
The Gauss-Seidel method for generalized Nash equilibrium problems of polynomials,
\emph{Computational Optimization and Applications}
78, 529--557, 2021.


\bibitem{Nie2020bilevel}
J. Nie, L. Wang, J. Ye and S. Zhong,
A Lagrange multiplier expression method for bilevel polynomial optimization,
\emph{SIAM Journal on Optimization}
31(3), 2368--2395, 2021.



\bibitem{Nie2020nash}
J.~Nie and X.~Tang,
Nash equilibrium problems of polynomials,
\emph{Preprint}, 2020.


\bibitem{Nie2021convex}
J.~Nie and X.~Tang,
Convex generalized Nash equilibrium problems and polynomial optimization,
\emph{Math. Program.} (2021).
\url{doi.org/10.1007/s10107-021-01739-7}


\bibitem{Nie2022degree}
J.~Nie, K.~Ranestad and X.~Tang,
Algebraic degrees of generalized Nash equilibrium problems,
\emph{Preprint}, \url{arXiv:2208.00357}, 2022.


\bibitem{Pang2005}
J.~Pang and M.~Fukushima,
Quasi-variational inequalities, generalized Nash equilibria, and multi-leader-follower games,
\emph{Computational Management Science}
2, 21--56, 2005.

\bibitem{Pang2008}
J.~Pang, G.~Scutari, F.~Facchinei and C.~Wang,
Distributed power allocation with rate constraints in Gaussian parallel interference channels,
\emph{IEEE Transactions on Information Theory}
54(8), 3471--3489, 2008.

\bibitem{Pang2011}
J.~Pang and G.~Scutari,
Nonconvex games with side constraints,
\emph{SIAM Journal on Optimization}
21(4), 1491-1522, 2011.

\bibitem{putinar1993positive}
M.~Putinar,
Positive polynomials on compact semi-algebraic sets,
\emph{Indiana University Mathematics Journal}
42(3), 969--984, 1993.


\bibitem{Schiro2013}
D.~Schiro, J.~Pang, and U.~Shanbhag,
On the solution of affine generalized Nash equilibrium problems
with shared constraints by Lemke's method,
\emph{Mathematical Programming}
142(1), 1--46, 2013.

\bibitem{Swg05}
M.~Schweighofer,
Optimization of polynomials on compact semialgebraic sets,
\emph{SIAM Journal on Optimization}
15(3), 805--825, 2005.



\bibitem{sturm1999using}
J.~Sturm,
Using SeDuMi 1.02, a MATLAB toolbox for optimization over symmetric cones.
\emph{Optimization Methods and Software}
11(1-4), 625--653, 1999.




\bibitem{vonHeusinger2009-2}
A.~von~Heusinger and C.~Kanzow,
Optimization reformulations of the generalized Nash equilibrium problem
using Nikaido-Isoda-type functions,
\emph{Computational Optimization and Applications}
43(3), 353--377, 2009.


\bibitem{Heusinger2012}
A.~von~Heusinger, C.~Kanzow, and M.~Fukushima,
Newton’s method for computing a normalized equilibrium
in the generalized Nash game through fixed point formulation,
\emph{Mathematical Programming}
132(1), 99--123, 2012.



\end{thebibliography}
\end{document}